\documentclass[11pt,a4paper,leqno]{amsart}
\usepackage{a4wide}
\usepackage{latexsym}
\usepackage{amsmath}
\usepackage{amssymb}
\usepackage{stackrel} 
\usepackage{color}
\usepackage{graphicx}
\usepackage{hyperref}
\usepackage{amsthm,amstext,units,enumerate}

\newtheorem{lemma}{Lemma}
\newtheorem{prop}{Proposition}
\newtheorem{theo}{Theorem}

\theoremstyle{definition}
\newtheorem{defin}{Definition}
\newtheorem{remark}{Remark}

\newcommand{\C}{\mathbb{C}}

\newcommand{\Kk}{\mathcal{K}}
\newcommand{\N}{\mathbb{N}}
\newcommand{\R}{\mathbb{R}}

\newcommand{\Oo}{\mathcal{O}}

\newcommand{\EE}{\mathbb{E}}
\newcommand{\MM}{\mathbb{M}}
\newcommand{\NN}{\mathbb{N}}
\newcommand{\RR}{\mathbb{R}}

\makeatletter
\def\author@andify{%
  \nxandlist {\unskip ,\penalty-1 \space\ignorespaces}%
    {\unskip {} \@@and~}%
    {\unskip \penalty-2 \space \@@and~}%
}
\makeatother

\begin{document}
\title[Summability formal solutions generalized moment integro-differential equations]{Summability of formal solutions for a family of generalized moment integro-differential equations}
\author{Alberto Lastra}
\address{Departamento de F\'isica y Matem\'aticas\\
University of Alcal\'a\\
Ap. de Correos 20, E-28871 Alcal\'a de Henares (Madrid), Spain}
\email{alberto.lastra@uah.es}
\author{S{\l}awomir Michalik}
\address{Faculty of Mathematics and Natural Sciences,
College of Science\\
Cardinal Stefan Wyszy\'nski University\\
W\'oycickiego 1/3,
01-938 Warszawa, Poland}
\email{s.michalik@uksw.edu.pl}
\urladdr{\url{http://www.impan.pl/~slawek}}
\author{Maria Suwi\'nska}
\address{Faculty of Mathematics and Natural Sciences,
College of Science\\
Cardinal Stefan Wyszy\'nski University\\
W\'oycickiego 1/3,
01-938 Warszawa, Poland}
\email{m.suwinska@op.pl}
\date{}
\begin{abstract}

Generalized summability results are obtained regarding formal solutions of certain families of linear moment integro-differential equations with time variable coefficients. The main result leans on the knowledge of the behavior of the moment derivatives of the elements involved in the problem.

A refinement of the main result is also provided giving rise to more accurate results which remain valid in wide families of problems of high interest in practice, such as fractional integro-differential equations.

\end{abstract}

\maketitle
\thispagestyle{empty}

\section{Introduction}

The present work deals with the summability properties of the formal solutions of certain families of generalized moment integro-differential equations described by the formula
\begin{equation}
\left(1-\sum_{i\in\Kk}\sum_{q=0}^{p_{i}}a_{iq}(z)\partial_{m_{1},t}^{-i}\partial_{m_{2},z}^{q}\right)u(t,z)=\hat{f}(t,z)\label{eq:mainintro},
\end{equation}
where $\mathcal{K}\subseteq\{1,\ldots,\kappa\}$ stands for a finite subset of natural numbers, $p_i\ge 0$ for all $i\in\mathcal{K}$, $a_{iq}(z)$ are holomorphic functions in a neighborhood of the origin and $\hat{f}$ is a formal power series in the variables $(t,z)$. Differentiation and integration operators given in the equation are of moment nature. These operators generalize the usual derivation and integration operators, respectively. 

Under certain assumptions to be precised later the first main result of the present work (Theorem~\ref{teopral}) states that the formal solution of (\ref{eq:mainintro}) is summable with respect to some strongly regular sequence (see Section~\ref{secsrs}) whenever $\hat{f}$ is summable with respect to that sequence. The equivalence is attained in a refinement of the first main result (Theorem~\ref{coropral}) under more restrictive assumptions on the strongly regular sequences involved. Such restrictive conditions are satisfied in practice in the framework of the most outstanding families of strongly regular sequences, namely when dealing with integro-differential equations (Gevrey settings) and also in the framework of fractional integro-differential equations (see Section~\ref{secscopeA}).

The concept of a moment derivative was first described by W. Balser and M. Yoshino in 2010, in the seminal work~\cite{BY}. Given a sequence of positive real numbers $m=(m(p))_{p\ge0}$, the operator $\partial_{m,z}:\mathbb{C}[[z]]\to\mathbb{C}[[z]]$ is defined by 
$$\partial_{m,z}\left(\sum_{p\ge0}\frac{f_p}{m(p)}z^p\right)=\sum_{p\ge0}\frac{f_{p+1}}{m(p)}z^p.$$
This definition generalizes the classical derivation of power series with positive radius of convergence, particularized to $m=(p!)_{p\ge0}$, and also Caputo $1/k$-fractional derivative $\partial_z^{1/k}$ if  $m:=(\Gamma(1+\frac{p}{k}))_{p\ge0}$ is considered. The second main result of the present study is satisfied in both situations.

The wide range of sequences belonging to the class of strongly regular sequences allow to study a variety of operators and therefore of functional equations under more generality. This fact has caused the increasing interest of the scientific community regarding two issues: the properties of the spaces of series and functions associated with strongly regular sequences (see for example~\cite{sanzproceedings,jjs17} and the references therein) and also the summability properties of formal solutions of functional equations in the complex domain. This work approaches the second direction. 

After the aforementioned work~\cite{BY}, the study of moment differential equations in the complex domain was followed by different studies. In~\cite{M}, the second author gives a solution of families of Cauchy problems regarding moment partial differential equations with respect to different moments associated with each variable. The convergence and summability of the formal solutions of homogeneous (resp. inhomogeneous) linear moment partial differential equations with constant coefficients are considered by the second author in~\cite{michalik13jmaa} (resp. in~\cite{michalik17fe}). We also refer to~\cite{lastramaleksanz} in this direction.

In addition to this, the theory of summability of formal solutions of generalized moment partial differential equations leans on the growth estimates of the coefficients. This concern has been discussed in~\cite{michaliksuwinska,suwinska,LMS}. Some problems regarding strongly regular sequences and other aspects quite related to the summability of formal solutions of functional equations have also been considered: the Stokes phenomenon in~\cite{michaliktkacz}, or the summability properties of the formal solutions of PDEs with coefficients whose growth is governed by a strongly regular sequence, in~\cite{lastramaleksanz2}.

So far, partial achievements have been accomplished regarding the summability of formal solutions of some families of generalized moment partial differential equations with variable coefficients. As a first step we cite~\cite{LMS2}, where the integral representation of functional moment derivatives allows to guarantee that the moment derivative of a summable function remains summable (see Theorem 3 and Corollary 1,~\cite{LMS2}). This fact is applied to obtain summability results for the Cauchy problem
\begin{equation}\label{e2intro}
\left\{ \begin{array}{rcl}
            \left(\partial_{m_1,t}^{\kappa}-a(z)\partial_{m_2,z}^{p}\right)u(t,z)&=&\hat{f}(t,z)\\
             \partial_{m_1,t}^{j}u(0,z)&=&\varphi_j(z),\quad j=0,\ldots,\kappa-1
             \end{array}
\right.,
\end{equation} 
where $a(z),\,\varphi_j(z)$ are holomorphic functions defined in a neighborhood of the origin, with $a(0)\neq 0$. This problem turns out to be the moment model of the main problem under study in~\cite{remy2016}, in the framework of moment partial differential equations. 

In the present study we give a further step in the theory of summability of generalized moment functional equations, by considering the moment integro-differential problem (\ref{eq:mainintro}), which can be considered as a moment interpretation of the integro-differential equations in~\cite{remy2017}.

The first main result, Theorem~\ref{teopral}, states that generalized summability of $\hat{f}(t,z)$  with respect to $t$ uniformly on $z$, together with the summability of a finite number of formal power series related to the formal solution (see condition (\ref{e498})) entail generalized summability of the formal solution of (\ref{eq:mainintro}). A more accurate result is also established in Theorem~\ref{coropral} under the more restrictive Assumption (A) regarding the strongly regular sequences involved. Assumption (A) can be read as a closeness condition of generalized summability with respect to certain moment integration. Under Assumption (A), condition (\ref{e498}) is satisfied, turning the result in Theorem~\ref{teopral} into an equivalence. The work ends with Section~\ref{secscopeA}, where two essential examples are displayed in which Assumption (A) holds. Indeed, these two situations are related to the case of classical integro-differential equations (see Proposition~\ref{prop732a}), and fractional integro-differential equations (see Proposition~\ref{prop732b}).

Some essential differences in the approach of both works are worth mentioning. On the one hand, the moment operators describing equation (\ref{eq:mainintro})  give more generality in the sense that the results obtained can be particularized into concrete problems regarding not only classical integro-differential equations but also integro-differential equations involving fractional derivatives. On the other hand, the coefficients considered in the present work remain constant in time, due to the absence of appropriate tools to deal with this concern in the moment settings. This deserves attention of a future investigation. 

Going back to the Cauchy problem (\ref{e2intro}), one can check that it can be transformed into a moment integro-differential equation of the form (\ref{eq:mainintro}), and therefore be seen as an example of equation for which the first main result of the present work applies.

The paper is structured as follows. The notation is fixed in Section~\ref{secnot}. Section~\ref{secpre} reviews the main concepts and properties related to strongly regular sequences, asymptotic expansions in a sectorial region of the complex plane associated with one of such sequences, and appropriate tools in order to achieve summability results in this framework. Section~\ref{secpre} concludes with the definition and the main properties of the Banach spaces of functions involved in the proof of the first main result. In section~\ref{secsum}, we first recall the construction of the Newton polygon and state the problem under study. Some technical parts in the proof of the first main result (Theorem~\ref{teopral}) are left to a final section, Section~\ref{secanexo}, for the sake of clarity of the reasoning. Assumption (A) leads to the second main result (Theorem~\ref{coropral}) and then to Section~\ref{secscopeA}, where the importance of Assumption (A) is put into light in several applications of the theory.

\section{Notation}\label{secnot}

$\mathbb{N}$ stands for the set of natural numbers $\{1,2,\cdots\}$ and $\mathbb{N}_0:=\mathbb{N}\cup\{0\}$. We write $\mathbb{Q}_+$ for the set of positive rational numbers.

The symbol $\left\lfloor \cdot \right\rfloor$ stands for the floor function.

We write $\mathcal{R}$ for the Riemann surface of the logarithm. 

For any $r>0$, $D(0,r)$ (resp. $\overline{D}(0,r)$) stands for the open (resp. closed) disc in the complex plane $\{z\in\C:|z|<r\}$ (resp. $\{z\in\C:|z|\le r\}$).

Let $\theta>0$ and $d\in\R$. $S_{d}(\theta)$ denotes the open infinite sector in $\mathcal{R}$ 
$$S_{d}(\theta):=\left\{z\in\mathcal{R}: |\hbox{arg}(z)-d|<\frac{\theta}{2}\right\}.$$
If the opening of the sector is unspecified we write $S_d$. A sectorial region $G_d(\theta)\subseteq\mathcal{R}$ is a set satisfying $G_d(\theta)\subseteq S_{d}(\theta)\cap D(0,r)$ for some $r>0$, and for all $0<\theta'<\theta$ there exists $0<r'<r$ with $(S_d(\theta')\cap D(0,r'))\subseteq G_d(\theta)$. We denote by $\hbox{arg}(S)$ the set of arguments of $S$, in particular $\hbox{arg}(S_{d}(\theta))=\left(d-\frac{\theta}{2},d+\frac{\theta}{2}\right)$.

We write $\hat{S}_{d}(\theta;r):=S_{d}(\theta)\cup D(0,r)$, and $\hat{S}_{d}(\theta)$ (resp. $\hat{S}_{d}$) if $r>0$ (resp. $r>0$ and $\theta>0$) are not specified. The symbol $S\prec S_d(\theta)$ describes an infinite sector $S$ with the vertex at the origin and $\overline{S}\subseteq S_d(\theta)$. Analogously, $\hat{S}\prec \hat{S}_d(\theta;r)$ means that $\hat{S}=S\cup D(0,r')$, with $S\prec S_{d}(\theta)$ and $0<r'<r$. Given two sectorial regions $G_d(\theta)$ and $G_{d'}(\theta')$, we write $G_d(\theta)\prec G_{d'}(\theta')$ in the case that the previous property holds for the sectors involved in the definition of the corresponding sectorial regions.

Given a complex Banach space $(\mathbb{E},\left\|\cdot\right\|_{\mathbb{E}})$, the set $\mathcal{O}(U,\mathbb{E})$ stands for the set of holomorphic functions in a set $U\subseteq\C$, with values in $\mathbb{E}$. If $\mathbb{E}=\C$, then we simply write $\mathcal{O}(U)$. We denote the formal power series with coefficients in $\mathbb{E}$ by $\mathbb{E}[[z]]$.


\section{Preliminary results and definitions}\label{secpre}

\subsection{Strongly regular sequences and related properties}\label{secsrs}

We first recall the concept of strongly regular sequence put forward by V. Thilliez~\cite{thilliez} together with some of the properties held by such sequences, which will be useful in the sequel. Let us consider an increasing sequence of non-negative real numbers
$\MM=\left(M_{p}\right)_{p\ge0}$, with $M_{0}=1$, such that: 
\begin{enumerate}
\item [$(lc)$] $M_{p}^{2}\leq M_{p-1}M_{p+1}$ for every $p\ge1$, 
\item [$(mg)$] $\exists_{A_{1}>0}:\quad M_{p+q}\le A_{1}^{p+q}M_{p}M_{q}$
for every $p,q\in\NN_{0}$,
\item [$(snq)$] $\exists_{A_{2}>0}:\quad\sum_{q\ge p}\frac{M_{q}}{(q+1)M_{q+1}}\le A_{2}\frac{M_{p}}{M_{p+1}}$
for every $p\in\NN_{0}$. 
\end{enumerate}
The notation $(lc)$ stands for logarithmically convex, $(mg)$ for moderate growth and $(snq)$ for strong non-quasianalyticity conditions, respectively.

\begin{defin}
Any sequence of positive real numbers $\MM=\left(M_{p}\right)_{p\ge0}$, with $M_{0}=1$ under the properties $(lc)$, $(mg)$ and $(snq)$ is known as a \textit{strongly regular sequence}.
\end{defin}

The occurrence of strongly regular sequences has had a predominant role in the study of formal solutions of functional equations and their summability. The most outstanding example of a strongly regular sequence is the Gevrey sequence of order $\alpha>0$, defined by $\MM_{\alpha}=(p!^{\alpha})_{p\ge0}$, playing a crucial role in the study of summability of ordinary and partial differential equations. Another classical example, which is a generalization of the previous one,  is that of the sequence defined by $\MM_{\alpha,\beta}=(p!^{\alpha}\prod_{m=0}^p\log^{\beta}(e+m))_{p\ge 0}$, for fixed $\alpha>0$ and $\beta\in\R$. These sequences appear in the study of formal solutions of difference equations for $\alpha=1$ and $\beta=-1$, in the so-called $1+$ level~\cite{immink,immink2}

The following statements are a direct consequence of $\MM$ being an $(lc)$ sequence:

\begin{lemma}\label{lem:decreasing}
The sequence defined by $\left(\frac{M_{p}}{M_{p+1}}\right)_{p\ge0}$ is monotone decreasing.
\end{lemma}

\begin{lemma}\label{lem:prodM}
 For every $p,\,q\in\NN_0$ we have $M_{p}M_{q}\le M_{p+q}$.
\end{lemma}

As a direct consequence of the previous property one derives monotonicity of $\MM$, after an innocuous modification of $M_1$, which does not vary the asymptotic behavior on the growth of the elements of the initial sequence.

\begin{lemma}
 For any $n,p,q\in\NN_0,\,q\neq 0$ we have
 $$
 \frac{M_{n+p}}{M_{n+p+q}}\le\frac{M_{p}}{M_{p+q}}.
 $$
\end{lemma}
\begin{proof}
 Let us notice that
 $$
 \frac{M_{n+p}}{M_{n+p+q}}=\frac{M_{n+p}}{M_{n+p+1}}\cdot\frac{M_{n+p+1}}{M_{n+p+2}}\cdot\ldots\cdot\frac{M_{n+p+q-1}}{M_{n+p+q}}.
 $$
 From Lemma \ref{lem:decreasing} it follows that
 $$
 \frac{M_{n+p}}{M_{n+p+q}}\le\frac{M_{p}}{M_{p+1}}\cdot\frac{M_{p+1}}{M_{p+2}}\cdot\ldots\cdot\frac{M_{p+q-1}}{M_{p+q}}=\frac{M_p}{M_{p+q}}.
 $$
\end{proof}

\begin{lemma}\label{lem:floor_quotient}

Let $p,q\in\NN$. Then: 
\begin{enumerate}
\item There exist constants $C_{1},D_{1}>0$ such that 
\[
M_{\left\lfloor \nicefrac{np}{q}\right\rfloor }\leq C_{1}D_{1}^{n}M_{n}^{\nicefrac{p}{q}}\textrm{ for every }n\in\NN_{0}.
\]
\item There exist constants $C_{2},D_{2}>0$ such that 
\[
M_{n}^{\nicefrac{p}{q}}\leq C_{2}D_{2}^{n}M_{\left\lfloor \nicefrac{np}{q}\right\rfloor }\textrm{ for every }n\in\NN_{0}.
\]
\end{enumerate}
\end{lemma}

\begin{proof}
From Lemma~\ref{lem:prodM}, and monotonicity of $\MM$ one gets
$$M_{\left\lfloor \nicefrac{np}{q}\right\rfloor }\stackrel{(q)}{\cdots}M_{\left\lfloor \nicefrac{np}{q}\right\rfloor }\le M_{\left\lfloor \nicefrac{np}{q}\right\rfloor q}\le M_{np}.$$
The first statement follows directly from here and the fact that $M_{np}\le A_1^{np}M_n^p$ . In order to give proof for the second statement, we make use of $(mg)$ condition together with Euclidean division which guarantees the existence of $j\in\{0,\ldots,q-1\}$ such that 
$$M_n\stackrel{(p)}{\cdots}M_n\le M_{np}=M_{\left\lfloor \nicefrac{np}{q}\right\rfloor q+j}\le M_jA_1^{np}M_{\left\lfloor \nicefrac{np}{q}\right\rfloor q}\le M_jA_1^{np+\left\lfloor \nicefrac{np}{q}\right\rfloor q}M_{\left\lfloor \nicefrac{np}{q}\right\rfloor }^q.$$
\end{proof}

In addition to this, we consider sequences $\MM=(M_p)_{p\ge 0}$ under the following condition:
\begin{equation}\label{eq:seq_property}
\forall_{d\in\NN}\,\exists_{C_{3}(d)>0}\,\forall_{n\in\NN}\quad\frac{M_{dn}}{M_{dn-1}}\leq C_{3}(d)\frac{M_{n}}{M_{n-1}}.
\end{equation}

\begin{remark}
Gevrey and $1+$ sequences satisfy the property given by \eqref{eq:seq_property}. Observe that for fixed $\alpha>0$ and $\beta\in\R$, the property in \eqref{eq:seq_property} holds for $C_3(d)=d^{\alpha}$ (resp. $C_3(d)=d^{\alpha}(1+\log(d))^{\beta}$) in the case of $\MM_{\alpha}$ (resp. $\MM_{\alpha,\beta}$).
In the latter, notice that
$$\left(\frac{\log(e+dn)}{\log(e+n)}\right)^{\beta}=\left(\frac{\log(e/d+n)+\log(d)}{\log(e+n)}\right)^{\beta}\le\left(\frac{\log(e+n)+\log(d)}{\log(e+n)}\right)^{\beta}\le (1+\log(d))^{\beta},$$
for all $n\in\NN$.
\end{remark}

\begin{defin} Let us consider a sequence $\left(M_{p}\right)_{p\ge0}$
satisfying properties listed above and let $s\in\RR.$ We call $\left(m(p)\right)_{p\ge0}$
\emph{an $M_{p}$-sequence of order $s$} if there exist positive
constants $A_{3},A_{4}$ such that 
\[
A_{3}\left(M_{p}\right)^{s}\le m(p)\le A_{4}\left(M_{p}\right)^{s}\textrm{ for every }p\in\NN_{0}.
\]
\end{defin}

\begin{defin} We call $\left(m(p)\right)_{p\ge0}$ \emph{a regular
$M_{p}$-sequence of order $s$} if it is an $M_{p}$-sequence and
moreover 
\[
A_{3}\left(\frac{M_{p}}{M_{p-1}}\right)^{s}\le \frac{m(p)}{m(p-1)}\le A_{4}\left(\frac{M_{p}}{M_{p-1}}\right)^{s}\textrm{ for every }p\in\NN.
\]

\end{defin}


$M_p$ sequences (resp. regular $M_p$ sequences) of fixed order have been previously considered by the authors in~\cite{LMS,LMS2} (resp.~\cite{LMS}) in the study of the estimates related to the formal solutions of moment partial differential equations, and the summability of such formal solutions. Examples of such sequences are provided in these previous works.

For any strongly regular sequence $\mathbb{M}=(M_p)_{p\ge0}$, one defines the function 
$$M(t):=\sup_{p\ge0}\log\left(\frac{t^p}{M_p}\right),$$ 
for $t>0$ and $M(0)=0$. The expression
$$\omega(\mathbb{M}):=\left(\limsup_{r\to\infty}\max\left\{0,\frac{\log(M(r))}{\log(r)}\right\}\right)^{-1},$$
defines a positive real number.

We remark that, given any positive number $s$ and a strongly regular sequence $\mathbb{M}=(M_p)_{p\ge0}$, the sequence $\overline{\mathbb{M}}=(M_p^s)_{p\ge0}$ remains a strongly regular sequence, with $\omega(\overline{\mathbb{M}})=s\omega(\mathbb{M})$.

\subsection{Generalized summability}

In this subsection, we recall the main points on the theory of generalized summability for the sake of completeness. We refer to~\cite{sanzproceedings,LMS2} and the references therein for further details on the topic. In the whole subsection, $\left(\EE,\|\cdot\|_{\EE}\right)$ denotes a fixed complex Banach space.

\begin{defin}
Let $\mathbb{M}=(M_p)_{p\ge0}$ be a strongly regular sequence, $\theta>0$ and $d\in\R$. Let $G_d(\theta)$ be a sectorial region and let $u(z)\in\mathcal{O}(G_d(\theta),\mathbb{E})$. We say that $u$ admits the formal power series $\hat{u}(z)=\sum_{p\ge0}a_pz^p\in\mathbb{E}[[z]]$ as its $\mathbb{M}$\textit{-asymptotic expansion} in $G_d(\theta)$ (at the origin) if for every $0<\theta'<\theta$, $S_d(\theta';r)\subseteq G_d(\theta)$ and every integer $N\ge 1$, there exist $C,A>0$ such that
$$\left\|u(z)-\sum_{p=0}^{N-1}a_pz^p\right\|_{\mathbb{E}}\le CA^NM_N|z|^{N},\quad z\in S_d(\theta';r).$$
\end{defin}

\begin{defin}\label{def249}
Let $\mathbb{M}=(M_p)_{p\ge0}$ be a strongly regular sequence with $\omega(\mathbb{M})>0$ and let $\hat{u}(z)=\sum_{p\ge0}a_pz^p\in\mathbb{E}[[z]]$ be a formal power series. We say that $\hat{u}$ is \textit{$\mathbb{M}$-summable} along direction $d\in\R$ if there exists $\theta>\omega(\mathbb{M})$, a sectorial region $G_d(\theta)$ and $u(z)\in\mathcal{O}(G_d(\theta),\mathbb{E})$  such that $u$ admits $\hat{u}$ as its $\mathbb{M}$-asymptotic expansion in $G_d(\theta)$.
\end{defin}

\begin{lemma}[Watson's lemma]
Under the assumptions of the previous definition, the function $u(z)\in\mathcal{O}(G_d(\theta),\mathbb{E})$ turns out to be unique, and it is known as the \textit{$\mathbb{M}$-sum} of $u$ along direction $d$. 
\end{lemma}

We now recall an equivalent definition of $\mathbb{M}$-summability. For that purpose, we first need to describe some formal and analytic operators.

The exponential growth at infinity is generalized from Gevrey settings to strongly regular sequences in terms of the function $M$.

\begin{defin}
Let $\mathbb{M}$ be a sequence of positive real numbers, $r,\theta>0$ and $d\in\R$. We write $\mathcal{O}^{\mathbb{M}}(\hat{S}_d(\theta;r),\mathbb{E})$ for the set of holomorphic functions $f$ defined in $\hat{S}_d(\theta;r)$ with values in $\mathbb{E}$ such that for all $0<\theta'<\theta$ and $0<r'<r$, there exist $\tilde{c},\tilde{k}>0$ such that
$$\left\|f(z)\right\|_{\mathbb{E}}\le\tilde{c}\exp\left(M(|z|/\tilde{k})\right),\quad z\in\hat{S}_d(\theta',r').$$
An analogous definition can be stated by substituting the set $\hat{S}_d(\theta',r')$ by $S_d(\theta')$.
\end{defin}

The methods of generalized summability lean on the existence of kernel functions under suitable properties that fit certain asymptotic necessities (see~\cite{jjs17,lastramaleksanz}). The following definition can be adapted to $\omega(\mathbb{M})\ge 2$ after an appropriate ramification. However, we remain in the case $\omega(\mathbb{M})<2$ for practical purposes.

\begin{defin}
Let $\mathbb{M}$ be a strongly regular sequence with $\omega(\mathbb{M})<2$. The functions $e,E$ define a pair of kernel functions for $\mathbb{M}$-summability under these properties:
\begin{itemize}
\item $e\in\mathcal{O}(S_0(\omega(\mathbb{M})\pi))$. Moreover, $e(z)/z$ is locally uniformly integrable at 0, i.e., there exists $t_0>0$, and for all $z_0\in S_{0}(\omega(\mathbb{M})\pi)$ there exists an open set $U$ such that $z_0\in U\subseteq S_0(\omega(\mathbb{M})\pi)$, with 
$$\int_0^{t_0}\frac{\sup_{z\in U}\left|e\left(\frac{t}{z}\right)\right|}{t}dt<\infty.$$
In addition to this, for all $\epsilon>0$ there exist $c, k>0$ such that
$$|e(z)|\le c\exp\left(-M\left(\frac{|z|}{k}\right)\right)\quad \hbox{ for all }z\in S_0(\omega(\mathbb{M})\pi-\epsilon),$$
and also $e(x)\in\R$ for all $x>0$.
\item $E\in\mathcal{O}(\C)$ and there exist $\tilde{c},\,\tilde{k}>0$ such that $\displaystyle |E(z)|\le \tilde{c}\exp\left(M\left(\frac{|z|}{\tilde{k}}\right)\right)$ for $z\in \C$. There exists $\beta>0$ such that for all $0<\tilde{\theta}<2\pi-\omega(\mathbb{M})\pi$ and $M_E>0$, there exists $\tilde{c}_2>0$ with 
$$|E(z)|\le \frac{\tilde{c}_2}{|z|^{\beta}},\quad z\in S_\pi(\tilde{\theta})\setminus D(0,M_E).$$
\item The \textit{moment function} associated with $e$, defined by
$$m_{e}(z):=\int_{0}^{\infty} t^{z-1}e(t)dt$$
is holomorphic in $\{z\in\C:\hbox{Re}(z)\ge 0\}$. The analytic expression of $E$ is given by
$$E(z)=\sum_{p\ge0}\frac{z^p}{m_{e}(p)},\quad z\in\C.$$
\end{itemize}
\end{defin}

Hereinafter, the sequence of moments defined from a moment function $m_e$ by $(m_e(p))_{p\ge 0}$ will be assumed to emerge from certain fixed pair of kernel functions that will be omitted in the sequel. In addition to this, we assume that the strongly regular sequence $\mathbb{M}$ in the previous construction admits a nonzero proximate order (see~\cite{sanzproceedings,lastramaleksanz}), which is not a restrictive assumption in practice, and moreover entails the existence of a pair of kernel functions for $\mathbb{M}$-summability.  We also recall that the sequence of moments $m_e:=(m_e(p))_{p\ge0}$ associated with a strongly regular sequence $\mathbb{M}$ and a pair of kernel functions is indeed a strongly regular sequence, with $\omega(\mathbb{M})=\omega(m_e)$. It is also worth remarking that for any positive $s$ the sequence $\overline{\mathbb{M}}=(M_p^s)_{p\ge0}$ is a strongly regular sequence which admits a pair of kernel functions for $\overline{\mathbb{M}}$-summability as long as $\overline{\mathbb{M}}=(M_p)_{p\ge0}$ is a strongly regular sequence which admits a pair of kernel functions for $\mathbb{M}$-summability.

\begin{defin}
 Given a sequence of moments $m_e=(m_e(p))_{p\ge0}$, the formal $m_e$-moment Borel transformation is defined by
$$\hat{\mathcal{B}}_{m_e,z}\left(\sum_{p\ge0}a_pz^p\right)=\sum_{p\ge0}\frac{a_p}{m_e(p)}z^p,$$ 
on $\mathbb{E}[[z]]$.
\end{defin}

The equivalent definition to Definition~\ref{def249} (see Theorem 6.18,~\cite{sanzproceedings}) is the following.

\begin{defin}
Let $\mathbb{M}$ be a strongly regular sequence which admits a nonzero proximate order. Let $m_e$ be a sequence of moments associated with $\mathbb{M}$. The series $\hat{u}\in\mathbb{E}[[z]]$ is $\mathbb{M}$-summable along direction $d\in\R$ if $\hat{\mathcal{B}}_{m_e,z}(\hat{u}(z))$ is a series with a positive radius of convergence, and the analytic function defining such series, say $u(z)$, can be extended to an infinite sector of bisecting direction $d$, say $\hat{S}_d$, with $u(z)\in\mathcal{O}^{\mathbb{M}}(\hat{S}_d,\mathbb{E})$.
\end{defin}

\begin{prop}[Proposition 6.20,~\cite{sanzproceedings} and Corollary 1,~\cite{LMS2}]\label{prop1}
The set of $\mathbb{M}$-summable series along direction $d\in\R$ is a differential algebra, which is additionally closed under $m_e$-differentiation.
\end{prop}

The use of an $\mathbb{M}$-analog to Laplace transform allows to describe $\mathbb{M}$-summability algorithmically.

\begin{prop}[see Section 6,~\cite{sanzproceedings}]\label{prop316}
Let $d\in\R$ and let $e,\,E$ be a pair of kernel functions for $\mathbb{M}$-summability. Let $\theta>0$. For every $f\in\mathcal{O}^{\mathbb{M}}(S_d(\theta),\mathbb{E})$, we define the $e$-\emph{Laplace transform} of $f$ along a direction $\tau\in\hbox{arg}(S_d(\theta))$ by
$$(T_{e,\tau}f)(z)=\int_0^{\infty(\tau)}e(u/z)f(u)\frac{du}{u},$$
for $|\hbox{arg}(z)-\tau|<\omega(\mathbb{M})\pi/2$, and $|z|$ small enough. One may vary $\tau\in\hbox{arg}(S_d)$  in order to define a holomorphic function $T_{e,d}f$ in a sectorial region $G_d(\theta+\omega(\mathbb{M})\pi)$.

If $\omega(\mathbb{M})<2$ let $G=G_d(\theta)$ be a sectorial region with $\theta>\omega(\mathbb{M})\pi$, given $f\in\mathcal{O}(G,\mathbb{E})$ continuous at $0$, and $\tau\in\R$ with $|\tau-d|<(\theta-\omega(\mathbb{M})\pi)/2$, the operator $T^{-}_{e,\tau}$, known as the $e$-\emph{Borel transform} along direction $\tau$ is defined by
$$(T^{-}_{e,\tau}f)(u):=\frac{-1}{2\pi i}\int_{\delta_{\omega(\mathbb{M})}(\tau)}E(u/z)f(z)\frac{dz}{z},\quad u\in S_{\tau},$$
where $S_{\tau}$ is an infinite sector of bisecting direction $\tau$ and small enough opening, and $\delta_{\omega(\mathbb{M})}(\tau)$ is the Borel-like path consisting of the concatenation of a segment from the origin to a point $z_0$ with $\hbox{arg}(z_0)=\tau+\omega(\mathbb{M})(\pi+\epsilon)/2$, for some small enough $\epsilon\in(0,\pi)$, followed with the arc of circle centered at 0, joining $z_0$ and the point $z_1$, with $\hbox{arg}(z_1)=\tau-\omega(\mathbb{M})(\pi+\epsilon)/2$, clockwise, and concluding with the segment of endpoints  $z_1$ and the origin.

Let $G_{d}(\theta)$ and $f$ be as above. The family $\{T^{-}_{e,\tau}\}_{\tau}$, with $\tau$ varying among the real numbers with $|\tau-d|<(\theta-\omega(\mathbb{M})\pi)/2$ defines a holomorphic function denoted by $T^{-}_{e,d}f$ in the sector $S_{d}(\theta-\omega(\mathbb{M})\pi)$ and $T^{-}_{e,d}f\in\mathcal{O}^{\mathbb{M}}(S_{d}(\theta-\omega(\mathbb{M})\pi),\mathbb{E})$.
\end{prop}

The following result describes a generalization of Theorem 30~\cite{balser} in the framework of strongly regular sequences, generalizing the result for Gevrey sequences. 

\begin{theo}\label{teo1}
Let $d\in\R$ and $\theta>0$. We also fix a strongly regular sequence $\mathbb{M}$ with positive $\omega(\mathbb{M})<2$ which admits a nonzero proximate order. Let $e$ and $E$ be a pair of kernel functions for $\mathbb{M}$-summability. Let $f\in\mathcal{O}^{\mathbb{M}}(S_d(\theta),\mathbb{E})$ and put $g(z)=(T_{e,d}f)(z)$, with $z\in G_d(\theta+\omega(\mathbb{M})\pi)$, a sectorial region with bisecting direction $d$ and opening $\theta+\omega(\mathbb{M})\pi$. Then, one has that $f\equiv T^{-}_{e,d}g$.
\end{theo}

Finally, we recall the definition of moment differentiation.

\begin{defin}
Given a sequence of moments $(m_e(p))_{p\ge0}$, the $m_e$-moment differentiation $\partial_{m_e,z}$ is defined on $\mathbb{E}[[z]]$ by
$$\partial_{m_e,z}\left(\sum_{p\ge0}\frac{a_p}{m_e(p)}z^p\right):=\sum_{p\ge0}\frac{a_{p+1}}{m_e(p)}z^p.$$
This definition can be naturally extended to holomorphic functions defined in a neighborhood of the origin, and also to the $\mathbb{M}$-sum of a formal power series (see~\cite{LMS2}, Definition 10).

The linear operator $\partial^{-1}_{m_e,z}$, the inverse of the moment derivative, is defined by 
$$\partial^{-1}_{m_e,z}(z^p)=\frac{m_e(p)}{m_e(p+1)}z^{p+1}$$
for every $p\ge0$, and formally extended to $\mathbb{E}[[z]]$.
\end{defin}

\subsection{Banach spaces of holomorphic functions}

In this subsection we set out to define and state related properties concerning the Banach spaces of holomorphic functions involved in the proof of the main result, Theorem~\ref{teopral}. By $\left(\EE,\|\cdot\|_{\EE}\right)$ we will henceforth denote a complex Banach space, and we fix a positive real number $r$. 

\begin{defin}\label{defi:norm}
Let $0<\tilde{r}< r$. We consider the space $\mathcal{O}(\overline{D}(0,r),\mathbb{E})$ of holomorphic functions in $\overline{D}(0,r)$, with values in $\mathbb{E}$, endowed with the norm given by
\[
\left\|\sum_{n=0}^{\infty}f_{n}z^{n}\right\|_{\tilde{r}}=\sum_{n=0}^{\infty}\left\|f_{n}\right\|_{\mathbb{E}}|z|^{n},\quad z\in\overline{D}(0,r),\,|z|=\tilde{r},
\]
for every $f(z)=\sum_{n=0}^{\infty}f_{n}z^{n}\in\mathcal{O}(\overline{D}(0,r),\mathbb{E})$.
\end{defin}

\begin{lemma}\label{rem:norm_product}
Let us take $f,g\in\Oo(\overline{D}(0,r),\mathbb{E})$ with $f(z)=\sum_{n=0}^{\infty} f_n z^n$ and $g(z)=\sum_{n=0}^{\infty} g_nz^n$. Then,
$$
\|f(z)g(z)\|_{\tilde{r}}\le \|f(z)\|_{\tilde{r}}\,\|g(z)\|_{\tilde{r}},\quad z\in\overline{D}(0,r),\,|z|=\tilde{r}< r.
$$
 
\end{lemma}
\begin{proof}
Let $0<\tilde{r}< r$. It is enough to note that
\begin{equation*}
 \|f(z)g(z)\|_{\tilde{r}}=\sum_{n=0}^{\infty}\left\|\sum_{k=0}^n f_k g_{n-k}\right\|_{\mathbb{E}}\tilde{r}^n \le \sum_{n=0}^{\infty}\sum_{k=0}^n \left\|f_k\right\|_{\mathbb{E}}\,\left\|g_{n-k}\right\|_{\mathbb{E}}\tilde{r}^{n}=\|f(z)\|_{\tilde{r}}\,\|g(z)\|_{\tilde{r}}.
\end{equation*}

\end{proof}

\begin{lemma}[Lemma 5,~\cite{LMS2}]\label{lem:integral}
Let $m(n)$ be a moment sequence and let $f\in\mathcal{O}\left(\overline{D}(0,r)\right)$.
If there exists $C<\infty$ and $n\in\NN_{0}$ such that 
\[
\|f(z)\|_{\tilde{r}}\leq C\frac{|z|^{n}}{m(n)}\textrm{ for every }z\in\overline{D}(0,r),\ |z|=\tilde{r}< r,
\]
then 
\[
\|\partial_{m,z}^{-k}f(z)\|_{\tilde{r}}\leq C\frac{|z|^{n+k}}{m(n+k)}\textrm{ for every }k\in\NN_{0}\textrm{ and }z\in\overline{D}(0,r),\ |z|=\tilde{r}.
\]
\end{lemma}

A generalization of the previous lemma is the following:

\begin{lemma}\label{lem:integral2} Suppose that $f(z)=f_{1}(z)+\dots+f_{p}(z)$
with $f_{j}\in\mathcal{O}(\overline{D}(0,r))$ for $j=1,\dots,p$ and 
\[
\|f_{j}(z)\|_{\tilde{r}}\leq C_{j}\frac{|z|^{n_{j}}}{m(n_{j})}\textrm{ for every }z\in\overline{D}(0,r),\ |z|=\tilde{r}< r,\ j=1,\dots,p.
\]
Then for any $k\in\NN_{0}$ we have 
\[
\|\partial_{m,z}^{-k}f(z)\|_{\tilde{r}}\leq\sum_{j=1}^{p}C_{j}\frac{|z|^{n_{j}+k}}{m(n_{j}+k)}\textrm{ for every }z\in\overline{D}(0,r),\ |z|=\tilde{r}.
\]
\end{lemma}
\begin{proof}
We use the previous lemma to all $f_{j}(z)$ to receive the following
\[
\|\partial_{m,z}^{-k}f(z)\|_{\tilde{r}}\le\sum_{j=1}^{p}\|\partial_{m,z}^{-k}f_{j}(z)\|_{\tilde{r}}\le\sum_{j=1}^{p}C_{j}\frac{|z|^{n_{j}+k}}{m(n_{j}+k)}.
\]
\end{proof}

\section{Summability of formal solutions of moment integro-differential equations with time variable coefficients}\label{secsum}

In this central section we achieve summability results of the formal solution of certain family of integro-differential equations in the complex domain. First, we state some preliminary geometric constructions regarding the Newton polygon associated with the integro-differential equations under study.

\subsection{The Newton polygon} The Newton polygon is a classical tool representing partial differential equations in a geometric fashion, introduced in~\cite{yonemura}. In the framework of linear moment partial differential equations this concept was put forward in~\cite{michalik17fe} (resp.~\cite{michaliksuwinska}), when dealing with PDEs with constant (resp. time-dependent) coefficients, also considered in~\cite{suwinska}. The definition provided here is a slightly modified version adapted to the framework of the family of linear integro-differential equations under consideration.

Let $\Kk\subset\left\{ 1,\dots,\kappa\right\} \subset\NN$ for certain
integer $\kappa\geq1$ and all $p_{i}\geq0$ for $i\in\Kk$. For all $i\in\Kk$ and $0\le q\le p_i$ we fix $a_{iq}(z)\in\mathcal{O}(D(0,r))$, for some positive $r$. 

Let $\mathbb{M}=(M_p)_{p\ge0}$ be a sequence of positive real numbers, and $m_1,m_2$ be regular $M_p$-sequences of positive orders $s_1,s_2$, respectively. 

We define 
\begin{equation}\label{e249}P(z,\partial_{m_1,t},\partial_{m_2,z})=1-\sum_{i\in\Kk}\sum_{q=0}^{p_i}a_{iq}(z)\partial_{m_1,t}^{-i}\partial_{m_2,z}^{q}.
\end{equation}

\begin{defin}
The Newton polygon associated with $P$ in (\ref{e249}) is given by
$$N(P,s_1,s_2)=\hbox{conv}\left\{\Delta(\kappa s_1,-\kappa)\cup\left(\bigcup_{i\in\Kk,q=0,\ldots,p_i}\Delta((\kappa-i)s_1+p_is_2,i-\kappa)\right)\right\},$$
where $\hbox{conv}\{\}$ stands for the convex hull and 
$$\Delta(a,b)=\{(x,y)\in\RR^2: x\le a, y\ge b\},\qquad (a,b)\in\RR^2.$$
\end{defin}

\subsection{Statement of the main problem}

 Let us consider a linear integro-differential equation of the form 
\begin{equation}
\left(1-\sum_{i\in\Kk}\sum_{q=0}^{p_{i}}a_{iq}(z)\partial_{m_{1},t}^{-i}\partial_{m_{2},z}^{q}\right)u(t,z)=\hat{f}(t,z)\label{eq:main}
\end{equation}
with $\Kk\subset\left\{ 1,\dots,\kappa\right\} \subset\NN$ for certain
integer $\kappa\geq1$ and all $p_{i}\geq0$ for $i\in\Kk$. Let $\mathbb{M}=(M_p)_{p\ge0}$ be a strongly regular sequence which satisfies (\ref{eq:seq_property}). Sequences $m_{1}=(m_1(p))_{p\ge0}$ and $m_{2}=(m_2(p))_{p\ge0}$ are regular $M_p$-sequences of positive orders $s_{1},s_{2}$,
respectively, i.e., there exist constants $0<a_{1},a_{2},b_{1},b_{2}<\infty$
such that for $j=1,2$ we have 
\begin{equation}
a_{j}\left(\frac{M_{n}}{M_{n-1}}\right)^{s_{j}}\leq\frac{m_{j}(n)}{m_{j}(n-1)}\leq b_{j}\left(\frac{M_{n}}{M_{n-1}}\right)^{s_{j}}\textrm{ for }n\in\NN.\label{eq:regular}
\end{equation}

All $a_{iq}(z)$ and moreover $\frac{1}{a_{\kappa p_{\kappa}}(z)}$ are assumed to be holomorphic functions in $z\in D(0,r)$, for some $0<r<1$, and $\hat{f}(t,z)\in\C[[t,z]]$. 

We also assume that the Newton polygon of the operator 
$$P(z,\partial_{m_1,t}^{-1},\partial_{m_2,z})=1-\sum_{i\in\Kk}\sum_{q=0}^{p_{i}}a_{iq}(z)\partial_{m_{1},t}^{-i}\partial_{m_{2},z}^{q}$$
shows one non-horizontal segment with positive slope $k$ given by the formula 
\begin{equation}
\frac{1}{k}=\frac{s_{2}p_{\kappa}-s_{1}\kappa}{\kappa}=\frac{s_{2}p_{\kappa}}{\kappa}-s_{1}.\label{eq:k-defi}
\end{equation}

Observe that (\ref{eq:k-defi}) entails 
$$
s_2p_{\kappa}>s_1\kappa.
$$
The Newton polygon associated with the operator $P$ is displayed in Figure~\ref{fig1}. Observe that the existence of a positive slope defining the Newton polygon can be read in terms of the elements involved in the definition of $P$ as follows.

\begin{remark}
There exists at least one $i\in\Kk$ such that we have $\frac{p_{i}}{i}>\frac{s_{1}}{s_{2}}$. Otherwise, one would have $(\kappa-i)s_1+p_is_2\le \kappa s_1$ for all $i\in\Kk$, which entails the absence of a positive slope within the Newton polygon.
\end{remark}

Moreover, from (\ref{eq:k-defi}) one also arrives at the following algebraic constraint. 

\begin{remark}\label{rem-p_k-p_i}

For all $i\in\Kk$ we have 
\begin{equation}
\frac{p_{i}}{i}\leq\frac{p_{\kappa}}{\kappa}.\label{eq:p_i-and-p_k}
\end{equation}
\end{remark}
\begin{proof}
Fix $i\in\Kk.$ Then 
\[
\frac{s_{2}p_{i}-s_{1}i}{i}\leq\frac{s_{2}p_{\kappa}-s_{1}\kappa}{\kappa},
\]
see Figure~\ref{fig1}. Hence, one gets that $s_{2}\kappa p_{i}-s_{1}\kappa i \leq s_{2}ip_{\kappa}-s_{1}\kappa i$, and consequently (\ref{eq:p_i-and-p_k}) holds.
\end{proof}

\begin{figure}
	\centering
		\includegraphics[width=0.45\textwidth]{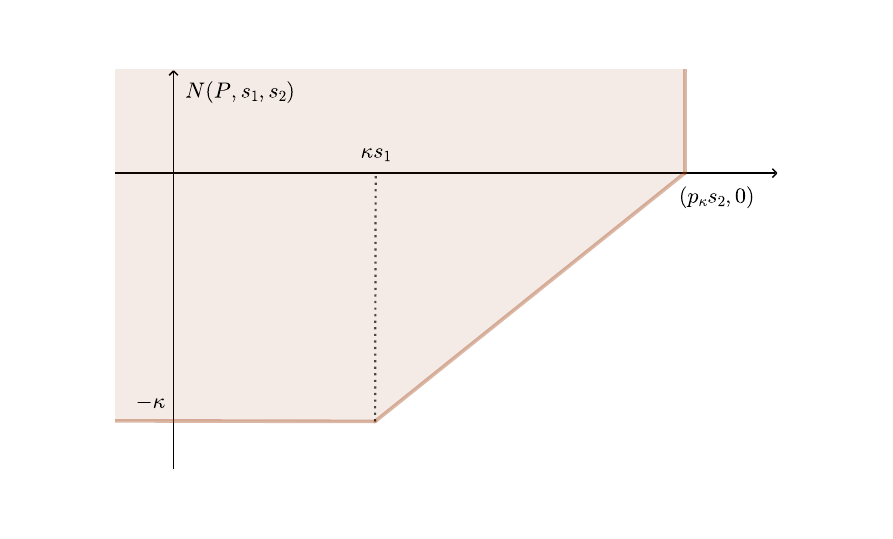}
		\includegraphics[width=0.45\textwidth]{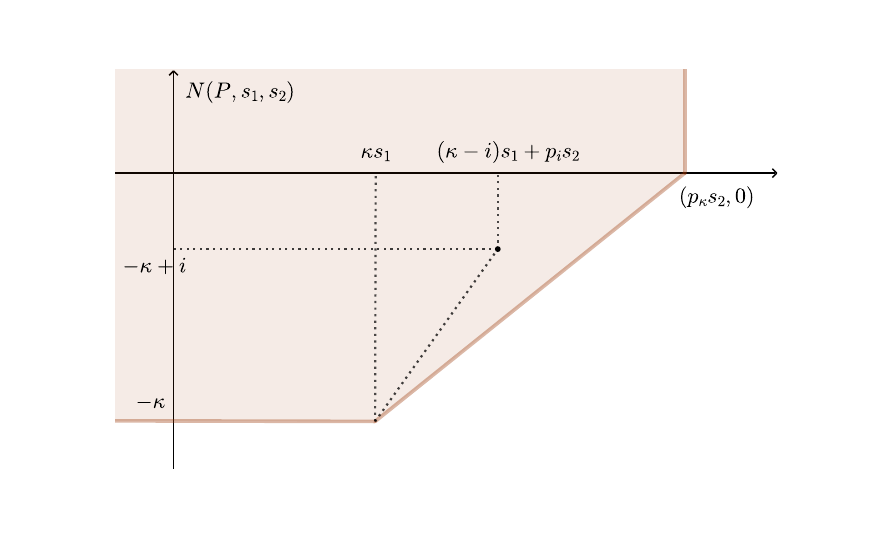}
		\caption{Newton polygon associated with $P$ (left) and geometric interpretation (right)}\label{fig1}
\end{figure}

\begin{prop}
The problem (\ref{eq:main}) admits a unique formal solution $\hat{u}(t,z)\in\mathbb{C}[[t,x]]$.
\end{prop}

\begin{proof}
Let us write $a_{iq}(z)=\sum_{p\ge0}a_{iq,p}z^p$, and $\hat{f}(t,z)=\sum_{p\ge0}\tilde{f}_p(t)z^p$. We put $\hat{u}(t,z)=\sum_{p\ge0}\tilde{u}_p(t)\frac{z^p}{m_2(p)}$ with $\tilde{f}_p,\tilde{u}_p\in\C[[t]]$. We plug the formal power series into the main equation to obtain
$$ \tilde{u}_{p}(t)-\sum_{i\in\Kk}\sum_{q=0}^{p_i}\sum_{l_1+l_2=p}a_{iq,l_1}\partial_{m_1,t}^{-i}\tilde{u}_{l_2+q}(t)m_2(p)/m_2(l_2)=\tilde{f}_p(t)m_2(p).$$
This entails that 
\begin{multline*}
a_{\kappa p_{\kappa},0}\partial_{m_1,t}^{-\kappa}\tilde{u}_{p+p_{\kappa}}(t)=\tilde{u}_p(t)-\tilde{f}_p(t)m_2(p)-\sum_{l_1+l_2=p,\,l_1>0}a_{\kappa p_{\kappa},l_1}\partial_{m_1,t}^{-\kappa}\tilde{u}_{l_2+p_{\kappa}}(t)m_2(p)/m_2(l_2)\\
-\sum_{i\in\Kk}\sum_{q\in Q_i}\sum_{l_1+l_2=p}a_{iq,l_1}\partial^{-i}_{m_1,t}\tilde{u}_{l_2+q}(t)m_2(p)/m_2(l_2),
\end{multline*}
with $Q_i=\{0,\ldots,p_i\}$ if $i<\kappa$ and $Q_{\kappa}=\{0,\ldots,p_{\kappa}-1\}$. We observe from the hypotheses made on the problem that $a_{\kappa p_{\kappa},0}\neq 0$, obtaining a unique formal power series, provided $\tilde{u}_{p}(t)$ are given for $p=0,\ldots,\kappa-1$.

\end{proof}

Hereinafter, $\mathbb{E}$ stands for the Banach space of holomorphic functions in $D(0,r)$, continuous up to $\overline{D}(0,r)$.
We are in conditions to state the main result. 

\begin{theo}\label{teopral}
Let $\overline{\mathbb{M}}=(M_p^{1/k})_{p\ge0}$, with $k$ given by (\ref{eq:k-defi}). Assume that $\hat{f}(t,z)\in\mathbb{E}[[t]]$ is $\overline{\mathbb{M}}$-summable along direction $d\in\R$. Moreover, assume that 
\begin{equation}\label{e498}
\partial_{m_1,t}^{-i}u_n(t), \qquad n\in\{0,\ldots,p_{\kappa}-1\},\quad i\in\Kk
\end{equation}
are $\overline{\mathbb{M}}$-summable along direction $d$. Then, it holds that the formal solution $\hat{u}(t,z)\in\mathbb{E}[[t]]$ of (\ref{eq:main}) is $\overline{\mathbb{M}}$-summable along direction $d$. 
\end{theo}

\begin{proof}
We observe that $\overline{\mathbb{M}}$ is a strongly regular sequence which admits a pair of kernel functions for $\overline{\mathbb{M}}$-summability. Let $e$ and $E$ be such kernel functions. 

Let $z\in \overline{D}(0,r)$. For $\hat{v}(t,z):=\partial_{m_{2},z}^{p_{\kappa}}\hat{u}(t,z)$
we have 
\[
\hat{u}(t,z)=\sum_{n=0}^{p_{\kappa}-1}u_{n}(t)z^{n}+\partial_{m_{2},z}^{-p_{k}}\hat{v}(t,z).
\]
If we then put $\hat{v}(t,z)=\partial_{m_{1},t}^{\kappa}\hat{w}(t,z)$ for a certain $\hat{w}(t,z)$
then (\ref{eq:main}) takes a form 
\begin{equation}
\label{eq:g}
\left(1-b_{\kappa p_{\kappa}}(z)\partial_{m_{1},t}^{\kappa}\partial_{m_{2},z}^{-p_{\kappa}}+\sum_{i\in\Kk}\sum_{q\in Q_{i}}b_{iq}(z)\partial_{m_{1},t}^{\kappa-i}\partial_{m_{2},z}^{q-p_{\kappa}}\right)\hat{w}(t,z)=\hat{g}(t,z),
\end{equation}
with 
\[
\hat{g}(t,z)=\frac{1}{a_{\kappa p_{\kappa}}(z)}\left(\sum_{n=0}^{p_{\kappa}-1}u_{n}(t)z^{n}-\sum_{i\in\Kk}\sum_{q=0}^{p_{i}}\sum_{n=0}^{p_{\kappa}-1}a_{iq}(z)\partial_{m_{1},t}^{-i}\partial_{m_{2},z}^{q}u_{n}(t)z^{n}-\hat{f}(t,z)\right).
\]
The formal power series $\hat{g}(t,z)\in\mathbb{E}[[t]]$ turns out to be $\overline{\mathbb{M}}$-summable along direction $d$, in view of the hypotheses made on $\hat{f}(t,z)$ and (\ref{e498}). Let us denote by $g(t,z)\in\mathcal{O}(G\times \overline{D}(0,r))$ its sum, with $G$ being a sectorial region with bisecting direction $d$ and opening larger than $\pi\omega(\overline{\mathbb{M}})$. In (\ref{eq:g}) we denote the set $\left\{ 0,1,\dots,p_{\kappa}-1\right\} $ by $Q_{\kappa}$  and put $Q_{i}=\left\{ 0,1,\dots,p_{i}\right\} $ for any $1<i<\kappa$. Note that all $b_{iq}$ are holomorphic in $\overline{D}(0,r)$ and moreover
\[
b_{iq}(z)=\left\{ \begin{array}{rl}
\frac{1}{a_{\kappa p_{\kappa}}(z)} & \textrm{for }i=\kappa,\,q=p_{\kappa}\\
\frac{a_{iq}(z)}{a_{\kappa p_{\kappa}}(z)} & \textrm{otherwise. }
\end{array}\right.
\]

We write $\hat{w}(t,z)=\sum_{p\geq0}\hat{w}_{p}(t,z)$ with 
\begin{equation}
\left\{ \begin{array}{rl}
\hat{w}_{0}(t,z) & =\hat{g}(t,z)\\
\displaystyle\hat{w}_{p+1}(t,z) & =\displaystyle\left(b_{\kappa p_{\kappa}}(z)\partial_{m_{1},t}^{\kappa}\partial_{m_{2},z}^{-p_{\kappa}}-\sum_{i\in\Kk}\sum_{q\in Q_{i}}b_{iq}(z)\partial_{m_{1},t}^{\kappa-i}\partial_{m_{2},z}^{q-p_{\kappa}}\right)\hat{w}_{p}(t,z)\textrm{ for }p\geq0.
\end{array}\right.\label{eq:w_p-rec1}
\end{equation}
and also 
$w(t,z)=\sum_{p\geq0}w_{p}(t,z)$ with 
\begin{equation}
\left\{ \begin{array}{rl}
w_{0}(t,z) & =g(t,z)\\
\displaystyle w_{p+1}(t,z) & =\displaystyle\left(b_{\kappa p_{\kappa}}(z)\partial_{m_{1},t}^{\kappa}\partial_{m_{2},z}^{-p_{\kappa}}-\sum_{i\in\Kk}\sum_{q\in Q_{i}}b_{iq}(z)\partial_{m_{1},t}^{\kappa-i}\partial_{m_{2},z}^{q-p_{\kappa}}\right)w_{p}(t,z)\textrm{ for }p\geq0.
\end{array}\right.\label{eq:w_p-rec}
\end{equation}

We observe that $w_p(t,z)\in\mathcal{O}(G\times \overline{D}(0,r))$, and $w_p$ admits $\hat{w}_p(t,z)\in\mathbb{E}[[t]]$ as its $\overline{\mathbb{M}}$-asymptotic expansion in $G$. 

Therefore, $w_p$ is the $\overline{\mathbb{M}}$-sum of $\hat{w}_p(t,z)$ in $G$.

Note that $w_{p}(t,z)$ can be bounded from above with use of Lemma \ref{lem:integral2}
for every $p\ge1$. Indeed, we can describe $w_{1}(t,z)$ as a
finite sum of holomorphic functions of the form $b_{iq}(z)\partial_{m_{1},t}^{\kappa-i}\partial_{m_{2},z}^{q-p_{\kappa}}w_{0}(t,z)$. More precisely,
\begin{equation}\label{e546}
w_{1}(t,z)=b_{\kappa p_{\kappa}}(z)\partial_{m_{1},t}^{\kappa}\partial_{m_{2},z}^{-p_{\kappa}}w_{0}(t,z)-\sum_{i\in\mathcal{K}}\sum_{q\in Q_{i}}b_{iq}(z)\partial_{m_{1},t}^{\kappa-i}\partial_{m_{2},z}^{q-p_{\kappa}}w_{0}(t,z).
\end{equation}
Additionally, we have $B=\max_{i,q}\max_{\tilde{r}\leq r}\left\Vert b_{iq}(z)\right\Vert _{\tilde{r}}$
 and
\begin{equation}\label{e546b}
\|\partial_{m_{1},t}^{n}w_{0}(t,z)\|_{\tilde{r}}\leq CK^{n}M_{n}^{\nicefrac{s_{2}p_{\kappa}}{\kappa}}\frac{|z|^{0}}{m_{2}(0)}\textrm{ for all }n\in\NN_{0},\ z\in\overline{D},\ |z|=\tilde{r}\ \textrm{and}\ t\in G.
\end{equation}
From these facts and Lemma \ref{lem:integral}
we conclude that

\begin{equation}\label{e546c}
\left\Vert b_{iq}(z)\partial_{m_{1},t}^{\kappa-i}\partial_{m_{2},z}^{q-p_{\kappa}}w_{0}(t,z)\right\Vert _{\tilde{r}}\leq BC'K'^{\kappa-i}M_{\left\lfloor \nicefrac{(\kappa-i)p_{\kappa}}{\kappa}\right\rfloor }^{s_{2}}\frac{|z|^{p_{\kappa}-q}}{m_{2}(p_{\kappa}-q)},
\end{equation}
for $|z|=\tilde{r}$ and $t\in G$. Hence,

\begin{align}
\|w_{1}(t,z)\|_{\tilde{r}}&\le\sum_{i\in\mathcal{K}\cup\left\{ 0\right\} }\sum_{q\in Q_{i}}\|b_{iq}(z)\partial_{m_{1},t}^{\kappa-i}\partial_{m_{2},z}^{q-p_{\kappa}}w_{0}(t,z)\|_{\tilde{r}}\nonumber\\
&\leq\sum_{i\in\mathcal{K}\cup\left\{ 0\right\} }\sum_{q\in Q_{i}}BC'K'^{\kappa-i}M_{\left\lfloor \nicefrac{(\kappa-i)p_{\kappa}}{\kappa}\right\rfloor }^{s_{2}}\frac{|z|^{p_{\kappa}-q}}{m_{2}(p_{\kappa}-q)}\label{e546d}
\end{align}
and
\begin{equation}\label{e546e}
\|\partial_{m_{1},t}^{n}w_{1}(t,z)\|_{\tilde{r}}\le\sum_{i\in\mathcal{K}\cup\left\{ 0\right\} }\sum_{q\in Q_{i}}BCK^{n+\kappa-i}M_{\left\lfloor \nicefrac{(n+\kappa-i)p_{\kappa}}{\kappa}\right\rfloor }^{s_{2}}\frac{|z|^{p_{\kappa}-q}}{m_{2}(p_{\kappa}-q)}.
\end{equation}
We can repeat this reasoning for all $p\geq1$ as seen in the proof of Lemma \ref{lem:w_p-bound}, which is postponed to Section~\ref{secanexo}.

\begin{lemma}\label{lem:w_p-bound}

Let $z\in D\left(0,r\right)$ with $0<r<1$. There exist constants $0<C',K',B'<\infty$ such that 
\begin{equation}
\left\Vert \partial_{m_{1},t}^{n}w_{p}(t,z)\right\Vert _{\tilde{r}}\leq B'^{p}C'K'^{n+\kappa p}M_{\left\lfloor \nicefrac{np_{\kappa}}{\kappa}\right\rfloor +p_{\kappa}p}^{s_{2}}P_{p}(|z|)\textrm{ for all }n\in\NN_{0},\ \tilde{r}=|z|,\label{eq:w_p}
\end{equation}
for all $t\in G$, where $P_{p}(z)$ is a polynomial given by the recursive formula

\[
\left\{\begin{array}{rl}
P_{0}(z) & =1\\
P_{p+1}(z) & =\displaystyle\left[\partial_{m_{2},z}^{-p_{\kappa}}+\sum_{i\in\Kk'}\sum_{q\in Q_{i}}\frac{M_{p_{\kappa}p}^{s_{2}}}{M_{p_{\kappa}p+p_{i}}^{s_{2}}}\partial_{m_{2},z}^{q-p_{\kappa}}\right]P_{p}(z)\textrm{ for all }p\ge0,
\end{array}\right.
\]
where $\Kk'=\{i\in\Kk:\ p_i\ge 1\}$.
\end{lemma}

\begin{lemma}\label{lema589}

For every $z\in D(0,r)$, with $0<r<1,$ and $p\in\NN_{0}$ we have
\begin{equation}
P_{p}(|z|)\le F^{p}\frac{|z|^{p}}{m_{2}(p_{\kappa}p)}\label{eq:P_p-bound}
\end{equation}
for certain positive constant $F$.
\end{lemma}

\begin{remark}
The above lemmas are the moment versions of Propositions 4 and 5, \cite{remy2017}.
\end{remark}

We combine Lemma~\ref{lem:w_p-bound} and Lemma~\ref{lema589} to receive 
\[
\left\Vert \partial_{m_{1},t}^{n}w_{p}(t,z)\right\Vert _{\tilde{r}}\leq B'^{p}C'K'^{n+\kappa p}M_{\left\lfloor \nicefrac{np_{\kappa}}{\kappa}\right\rfloor +p_{\kappa}p}^{s_{2}}F^{p}\frac{|z|^{p}}{m_{2}(p_{\kappa}p)}\textrm{ for all }n,p\in\NN_{0},\ \tilde{r}=|z|,
\]
and $t\in G$. Using Lemma \ref{lem:floor_quotient} we conclude that there exist
constants $C'',K''>0$ such that

\[
\left\Vert \partial_{m_{1},t}^{n}w_{p}(t,z)\right\Vert _{\tilde{r}}\leq B'^{p}C''K''^{n+\kappa p}M_{n+\kappa p}^{\nicefrac{s_{2}p_{\kappa}}{\kappa}}F^{p}\frac{|z|^{p}}{m_{2}(p_{\kappa}p)}\textrm{ for all }n,p\in\NN_{0},\ \tilde{r}=|z|,\ t\in G.
\]

From this it follows that 
\[
\sum_{p\ge0}\left\Vert \partial_{m_{1},t}^{n}w_{p}(t,z)\right\Vert _{\tilde{r}}\le C''K''^{n}M_{n}^{\nicefrac{s_{2}p_{\kappa}}{\kappa}}\sum_{p\ge0}\frac{M_{n+\kappa p}^{\nicefrac{s_{2}p_{\kappa}}{\kappa}}}{M_{n}^{\nicefrac{s_{2}p_{\kappa}}{\kappa}}m_{2}(p_{\kappa}p)}\left(B'K''^{\kappa}F|z|\right)^{p},
\]
for all $t\in G$. From Lemma \ref{lem:floor_quotient} and condition (lc)
it can be concluded that there exists a positive constant $\tilde{A}_{2}$
for which the inequality 
\[
\frac{M_{n+\kappa p}^{\nicefrac{s_{2}p_{\kappa}}{\kappa}}}{M_{n}^{\nicefrac{s_{2}p_{\kappa}}{\kappa}}m_{2}(p_{\kappa}p)}\le\tilde{A}_{2}^{n+\kappa p}
\]
holds. Hence, 
\begin{equation}
\sum_{p\ge0}\left\Vert \partial_{m_{1},t}^{n}w_{p}(t,z)\right\Vert _{\tilde{r}}\le C''K''^{n}\tilde{A}_{2}^{n}M_{n}^{\nicefrac{s_{2}p_{\kappa}}{\kappa}}\sum_{p\ge0}\left(B'K''^{\kappa}\tilde{A}_{2}^{\kappa}F|z|\right)^{p}.\label{eq:sum-w_p-bound}
\end{equation}

Let us now put $\rho=\min\left\{ r,\ B'^{-1}K''^{-\kappa}\tilde{A}_{2}^{-\kappa}F^{-1}\right\} $.
Then for all $z\in D(0,\rho)$ the last sum on the right-hand side
of (\ref{eq:sum-w_p-bound}) is finite and there exists a constant
$L>0$ such that 
\[
\sum_{p\ge0}\left(B'K''^{\kappa}\tilde{A}_{2}^{\kappa}F|z|\right)^{p}\leq L\textrm{ for }|z|<\rho.
\]
From this it follows that 
\[
\sum_{p\ge0}\left\Vert \partial_{m_{1},t}^{n}w_{p}(t,z)\right\Vert _{\tilde{r}}\le C''LK''^{n}\tilde{A}_{2}^{n}M_{n}^{\nicefrac{s_{2}p_{\kappa}}{\kappa}}\textrm{ for }z\in D(0,\rho),\ |z|=\tilde{r}.
\]

If we now put $\tilde{C}=C''L$ and $\tilde{K}=K''\tilde{A}_{2}$
then we can conclude that 
\begin{equation}\label{e667}
\left\Vert \partial_{m_{1},t}^{n}w(t,z)\right\Vert _{\tilde{r}}\le\tilde{C}\tilde{K}M_{n}^{\nicefrac{s_{2}p_{\kappa}}{\kappa}}\textrm{ for }z\in D(0,\rho),\ |z|=\tilde{r},\ n\in\NN_{0}.
\end{equation}

At this point, we prove that $w(t,z)$ is indeed the $\overline{\mathbb{M}}$-sum of $\hat{w}(t,z)$ defined by
$$\hat{w}(t,z)=\sum_{p\ge0}\hat{w}_p(t,z)\in\mathbb{E}[[t]]$$
following direction $d$. Let us choose a kernel function for $\overline{\mathbb{M}}$-summability, say $e$. One has that 
$$w_p(t,z)=T_{e,d}\hat{\mathcal{B}}_{m_e,t}\hat{w}_p(t,z),\qquad \omega(t,z)=\sum_{q\ge0}T_{e,d}\hat{\mathcal{B}}_{m_e,t}\hat{w}_p(t,z).$$
In view of (\ref{e667}), it follows that $T^{-}_{e,d}w(t,z)\in\mathcal{O}(D(0,r_1)\times \overline{D}(0,r))$, for some positive $r_1$. In view of Proposition~\ref{prop316}, we have that $T^{-}_{e,d}w(t,z)\in\mathcal{O}^{\overline{\mathbb{M}}}(S_d,\mathbb{E})$, for some infinite sector $S_d$ with bisecting direction $d$. Therefore, $T^{-}_{e,d}w(t,z)\in\mathcal{O}^{\overline{\mathbb{M}}}(\hat{S}_d,\mathbb{E})$.
The convergence of the series defining $w(t,z)$ together with Theorem~\ref{teo1} yield
$$T^{-}_{e,d}w(t,z)=T^{-}_{e,d}\sum_{p\ge0}T_{e,d}\hat{\mathcal{B}}_{m_e,t}(\hat{w}_p(t,z))=T^{-}_{e,d}T_{e,d}\sum_{p\ge0}\hat{\mathcal{B}}_{m_e,t}(\hat{w}_p(t,z))=\hat{\mathcal{B}}_{m_e,t}\hat{w}(t,z).$$
Therefore $\hat{\mathcal{B}}_{m_e,t}\hat{w}(t,z)\in\mathcal{O}^{\overline{\mathbb{M}}}(\hat{S}_d\times \overline{D}(0,r))$ concluding that $\hat{w}(t,z)$ is $\overline{\mathbb{M}}$-summable along direction $d$, with sum given by $w(t,z)$. 

We conclude $\overline{\mathbb{M}}$-summability of $u(t,z)\in\mathbb{E}[[t]]$ by observing that 
$$\hat{u}(t,z)=\sum_{n=0}^{p_{\kappa}-1}u_n(t)z^n+\partial_{m_2,z}^{-p_{\kappa}}\partial_{m_1,t}^{\kappa}\hat{w}(t,z),$$
and the fact that $u_n(t)$ for $n=0,\ldots,p_{\kappa}-1$ and $\partial_{m_2,z}^{-p_{\kappa}}\partial_{m_1,t}^{\kappa}\hat{w}(t,z)$ are $\overline{\mathbb{M}}$-summable along direction $d$.
\end{proof}

A more accurate result can be stated under Assumption (A):

\vspace{0.3cm}

\noindent \textbf{Assumption (A)}: The set of $\overline{\mathbb{M}}$-summable formal power series along direction $d$ is closed under the action of the operator $\partial_{m_1,t}^{-1}$, i.e., given a complex Banach space $\mathbb{F}$ and $\hat{u}(t)\in\mathbb{F}[[t]]$ which is $\overline{\mathbb{M}}$-summable along direction $d\in\R$, let $u(t)$ be its $\overline{\mathbb{M}}$-sum along direction $d$, then the formal power series $\partial_{m_1,t}^{-1}\hat{u}(t)$ is $\overline{\mathbb{M}}$-summable along direction $d\in\R$, whose $\overline{\mathbb{M}}$-sum along direction $d$, say $F_{d}(t)$, satisfies that $\partial_{m_1,t}F_{d}(t)$ coincides with $u(t)$.

\vspace{0.3cm}

Here, the application of $\partial_{m_1,t}$ to the $\mathbb{\overline{M}}$-sum of a formal power series is understood as in~\cite{LMS2}, Definition 10. Observe that in that preceding work, the correctness and suitability of such definition has been discussed.

\begin{defin}[Definition 10,~\cite{LMS2}]
Let $\mathbb{F}$ be a Banach space. Let $\mathbb{M}$ be a strongly regular sequence admitting a nonzero proximate order. Assume that $\hat{u}(t)\in\mathbb{F}[[t]]$ is $\mathbb{M}$-summable along direction $d\in\R$, with $\mathbb{M}$-sum given by $u(t)$. Let $m$ be a sequence of moments. The $m$-moment differentiation of $u(t)$ is defined by
$$\partial_{m,t}(u):=w,$$
with $w$ being the $\mathbb{M}$-sum of the formal power series $\partial_{m,t}(\hat{u}(t))\in\mathbb{F}[[t]]$ along direction $d$.
\end{defin}

A first direct consequence on Assumption (A) is that a recursion argument can be followed to guarantee that, under this assumption, one has that given $\hat{u}(t)\in\mathbb{F}[[t]]$ which is $\overline{\mathbb{M}}$-summable along direction $d\in\R$, with $\overline{\mathbb{M}}$-sum along direction $d$ given by $u(t)$, then for all $i\in\N$ the formal power series $\partial_{m_1,t}^{-i}\hat{u}(t)$ is $\overline{\mathbb{M}}$-summable along direction $d\in\R$, whose $\overline{\mathbb{M}}$-sum along direction $d$, say $F_{d,i}(t)$, satisfies that $\partial_{m_1,t}^{i}F_{d,i}(z)$ coincides with $u(t)$.

Observe that under Assumption (A) the condition (\ref{e498}) is reduced to assumption that $u_n(t)$ is $\overline{\mathbb{M}}$-summable along direction $d\in\R$ for $n\in\{0,\ldots,p_{\kappa}-1\}$.

\begin{theo}\label{coropral}
Under the hypotheses made on Theorem~\ref{teopral} and Assumption (A) the following statements are equivalent:
\begin{itemize}
\item[1.]  $\hat{f}(t,z)\in\mathbb{E}[[t]]$ and $u_n(t)\in\C[[t]]$ for $n\in\{0,\ldots,p_{\kappa}-1\}$ are $\overline{\mathbb{M}}$-summable along direction $d$.
\item[2.] The formal solution of (\ref{eq:main}) $\hat{u}(t,z)\in\mathbb{E}[[t]]$ is $\overline{\mathbb{M}}$-summable along direction $d$, and its sum is an analytic solution of (\ref{eq:main}), with the formal power series $\hat{f}(t,z)$ replaced by its sum $f(t,z)$.
\end{itemize} 
\end{theo}
\begin{proof}
The first part of the first implication ($1.\Rightarrow 2.$) is a direct consequence of Theorem~\ref{teopral} together with Assumption (A). For the second, we consider the function 
$$t\mapsto u(t,z)-\sum_{i\in\Kk}\sum_{q=0}^{p_i}a_{iq}(z)\partial_{m_2,z}^{q}U_i(t,z)-f(t,z),$$
where $U_i(t,z)$ stands for a function satisfying $\partial_{m_1,t}^{i}U_i(t,z)=u(t,z)$. Such function admits the null power series as its $\overline{\mathbb{M}}$-asymptotic expansion in a sector of bisecting direction $d$ and opening larger than $\omega(\overline{\mathbb{M}})\pi$. Watson's lemma yields the conclusion.

The implication ($2.\Rightarrow 1.$) is a consequence of the assumption (A), and the usual properties of the set of functions admitting asymptotic expansion in a sector.

\end{proof}

\subsection{On the scope of Assumption (A)}\label{secscopeA}

We conclude the work by providing concrete examples of families of strongly regular sequences in which Assumption (A) holds. Therefore, the stronger version of Theorem~\ref{teopral}, Theorem~\ref{coropral}, holds when dealing with such sequences.  

The most outstanding family of strongly regular sequences satisfying Assumption (A) is that of Gevrey sequences of a fixed positive order. The spaces of functions whose growth is related to this behavior is essential in the study of differential equations.

\begin{prop}\label{prop732a} 
Let $\alpha>0$. Let $m=(\Gamma(1+p))_{p\ge0}$ and consider the sequence $\overline{\mathbb{M}}:=\mathbb{M}_{\alpha}=(p!^{\alpha})_{p\ge0}$. Then, Assumption (A) is satisfied.
\end{prop}

The proof of the previous result can be found in Theorem 20~\cite{balser}. 

A second example of strongly regular sequences satisfying Assumption (A) is related to fractional derivatives and its application to the formal and analytic solutions of fractional partial differential equations. We first give some details about such operators and equations, which have been previously applied to such equations in~\cite{michalik10,michalik12,kilbas10}, and also refer to~\cite{kilbasetal06} and the references therein in order to deepen into this theory.

\begin{defin}
Let $\alpha\in\mathbb{Q}_+$. The fractional derivative of order $\alpha$ is the formal operator $\partial_z^{\alpha}:\mathbb{C}[[z^{\alpha}]]\to\mathbb{C}[[z^{\alpha}]]$, defined by
\begin{equation}\label{e739}
\partial_z^{\alpha}\left(\sum_{p\ge0}\frac{a_p}{\Gamma(1+\alpha p)}z^{\alpha p}\right)=\sum_{p\ge 0}\frac{a_{p+1}}{\Gamma(1+\alpha p)}z^{\alpha p}.
\end{equation}
\end{defin}

Observe that the formal fractional derivative turns out to be the usual derivative in the case that $\alpha=1$. Moreover, the fractional derivative of order $1/k$, for some fixed $k>0$, turns out to be the Caputo fractional derivative for $1/k$-analytic functions (see~\cite{kilbasetal06}) , applying the theory to fractional partial differential equations. The previous formal operator can be related to the moment differentiation in the following way. Given $k>0$, one has 
\begin{equation}\label{e748}
(\partial_{m_{1/k},z}\hat{f})(z^{1/k})=\partial_z^{1/k}(\hat{f}(z^{1/k})),
\end{equation}
with $m_{1/k}=(\Gamma(1+p/k))_{p\ge0}$. 

In view of (\ref{e739}), the formal $\alpha$-integral operator should satisfy
$$\partial_z^{-\alpha}z^{\alpha p}=\frac{\Gamma(1+\alpha p)}{\Gamma(1+\alpha(p+1))}z^{\alpha(p+1)}$$
for all $p\in\N_0$. Such property is satisfied by the Riemann-Liouville fractional integral defined by
\begin{equation}\label{e756}
I^{\alpha}_{0+}f(z):=\frac{1}{\Gamma(\alpha)}\int_0^z\frac{f(t) dt}{(z-t)^{1-\alpha}}.
\end{equation}
Indeed, for all $\beta>-1$ one has
$$I^{\alpha}_{0+}z^{\beta}=\frac{\Gamma(1+\beta)}{\Gamma(1+\alpha+\beta)}z^{\alpha+\beta},$$
see (2.2.1) and (2.2.10) in~\cite{kilbasetal06}.

The next result generalizes Proposition~\ref{prop732a}.

\begin{prop}\label{prop732b} Let $k>0$. Let $m=m_{1/k}=\Gamma(1+p/k))_{p\ge0}$ and consider a strongly regular sequence $\overline{\mathbb{M}}$. Then, Assumption (A) is satisfied.
\end{prop}

\begin{proof} Let us write $\overline{\mathbb{M}}=(\overline{M}_p)_{p\ge0}$. Let $G$ be a sectorial region with vertex at the origin and opening larger than $\pi\omega(\overline{\mathbb{M}})$, and fix a sector $T\subseteq G$. Let $f\in\mathcal{O}(G,\mathbb{F})$ and assume that $f$ is the $\overline{\mathbb{M}}$-sum of the formal power series $\hat{f}(z)=\sum_{p\ge0}\frac{f_p}{\Gamma\left(1+\frac{p}{k}\right)}z^p\in\mathbb{F}[[z]]$ in $G$. Taking into account (\ref{e748}) and (\ref{e756}), we estimate
\begin{equation}\label{e770}
\left\|\frac{1}{\Gamma\left(\frac{1}{k}\right)}\int_0^{z}\frac{f(t^{1/k})}{(z-t)^{1-\frac{1}{k}}}dt-\sum_{j=1}^{N-1}\frac{f_{j-1}}{\Gamma\left(1+\frac{j}{k}\right)}z^{j/k}\right\|_{\mathbb{F}},
\end{equation}
for any $z\in T$ and $N\ge 2$. 
Following Appendix B~\cite{balser}, we write
$$z^{j/k}=\int_0^z\frac{(z-t)^{\frac{1}{k}-1}t^{\frac{j-1}{k}}}{B\left(\frac{1}{k},\frac{j-1}{k}+1\right)},$$
where $B(\cdot,\cdot)$ stands for the so-called Beta Integral, with
$$B\left(\frac{1}{k},\frac{j-1}{k}+1\right)=\frac{\Gamma\left(\frac{1}{k}\right)\Gamma\left(\frac{j-1}{k}+1\right)}{\Gamma\left(\frac{j}{k}+1\right)}.$$
Therefore, the expression in (\ref{e770}) can be rewritten in the form
\begin{equation}\label{e779}
\left\|\frac{1}{\Gamma\left(\frac{1}{k}\right)}\int_0^{z}\frac{t^{-\frac{N-1}{k}}\left(f(t^{1/k})-\sum_{j=1}^{N-1}\frac{f_{j-1}}{\Gamma\left(\frac{j-1}{k}+1\right)}t^{\frac{j-1}{k}}\right)}{t^{-\frac{N-1}{k}}(z-t)^{1-\frac{1}{k}}}\right\|_{\mathbb{F}}.
\end{equation}
On the one hand, the existence of an asymptotic expansion yields the existence of $\tilde{C},\tilde{A}>0$ such that
\begin{align*}
\left\|t^{-\frac{N-1}{k}}\left(f(t^{1/k})-\sum_{j=1}^{N-1}\frac{f_{j-1}}{\Gamma\left(\frac{j-1}{k}+1\right)}t^{\frac{j-1}{k}}\right)\right\|_{\mathbb{F}}&=\left\|t^{-\frac{N-1}{k}}\left(f(t^{1/k})-\sum_{j=0}^{N-2}\frac{f_{j}}{\Gamma\left(\frac{j}{k}+1\right)}t^{\frac{j}{k}}\right)\right\|_{\mathbb{F}}\\
&\le \tilde{C}\tilde{A}^{N-1}\overline{M}_{N-1}.
\end{align*}
On the other hand, the change of variable $t=z s$ yields
$$\int_0^z\frac{t^{\frac{N-1}{k}}}{(z-t)^{1-\frac{1}{k}}}dt=\int_0^1\frac{z^{\frac{N-1}{k}}s^{\frac{N-1}{k}}z}{z^{1-\frac{1}{k}}(1-s)^{1-\frac{1}{k}}}ds=z^{\frac{N}{k}}\int_0^1\frac{s^{\frac{N-1}{k}}}{(1-s)^{1-\frac{1}{k}}}ds=z^{\frac{N}{k}}\frac{\Gamma\left(\frac{1}{k}\right)\Gamma\left(1+\frac{N-1}{k}\right)}{\Gamma\left(1+\frac{N}{k}\right)},$$
for all $z\in T$. 
This entails that (\ref{e779}) can be upper estimated by
$$\tilde{C}\tilde{A}^{N-1}\overline{M}_{N-1}|z|^{\frac{N}{k}}\frac{\Gamma\left(1+\frac{N-1}{k}\right)}{\Gamma\left(1+\frac{N}{k}\right)}\le \tilde{C}_1\tilde{A}_1^{N}\overline{M}_{N}|z|^{\frac{N}{k}},$$
for some appropriate $\tilde{C}_1,\tilde{A}_1>0$. This concludes the result.
\end{proof}

\section{Auxiliary results}\label{secanexo}
In this final section we collect the proofs of some auxiliary lemmas involved in that of Theorem~\ref{teopral}, which have been set aside in order not to interfere with the reasonings, and for the sake of clarity. These lemmas are inspired by Propositions 4 and 5, \cite{remy2017}.

\begin{proof}[Proof of Lemma~\ref{lem:w_p-bound}]
We observe that (\ref{eq:w_p}) holds true for $p=0$ if one takes into consideration the initial recursion argument involving (\ref{e546})--(\ref{e546e}) together with Lemma \ref{lem:floor_quotient}. Let us suppose that the statement is true for a fixed $p\ge0$. Then, putting
 $Q_{0}=\left\{ 0\right\} $, we get 
\begin{align*}
\left\Vert \partial_{m_{1},t}^{n}w_{p+1}(t,z)\right\Vert _{\tilde{r}} & \leq B\sum_{i\in\Kk\cup\left\{ 0\right\} }\sum_{q\in Q_{i}}\left\Vert \partial_{m_{1},t}^{n+\kappa-i}\partial_{m_{2},z}^{q-p_{\kappa}}w_{p}(t,z)\right\Vert _{\tilde{r}}\\
 & \leq B'^{p}C'K'^{n+\kappa p+\kappa}B\sum_{i\in\Kk\cup\left\{ 0\right\} }\sum_{q\in Q_{i}}K'^{-i}M_{\left\lfloor \nicefrac{\left(n+\kappa-i\right)p_{\kappa}}{\kappa}\right\rfloor +p_{\kappa}p}^{s_{2}}\partial_{m_{2},z}^{q-p_{\kappa}}(P_{p})(|z|).
\end{align*}
For $K'\ge 1$ we get 
\begin{multline*}
\left\| \partial_{m_{1},t}^{n}w_{p+1}(t,z)\right\| _{\tilde{r}}\\
\leq B'^{p}C'K'^{n+\kappa p+\kappa}BM_{\left\lfloor \nicefrac{np_{\kappa}}{\kappa}\right\rfloor +p_{\kappa}+p_{\kappa}p}^{s_{2}}\sum_{i\in\Kk\cup\left\{ 0\right\} }\sum_{q\in Q_{i}}\frac{M_{\left\lfloor \nicefrac{\left(n-i\right)p_{\kappa}}{\kappa}\right\rfloor +p_{\kappa}+p_{\kappa}p}^{s_{2}}}{M_{\left\lfloor \nicefrac{np_{\kappa}}{\kappa}\right\rfloor +p_{\kappa}+p_{\kappa}p}^{s_{2}}}\partial_{m_{2},z}^{q-p_{\kappa}}(P_{p})(|z|).
\end{multline*}
From Remark \ref{rem-p_k-p_i} and Lemma \ref{lem:decreasing} we
conclude that 
\[
\frac{M_{\left\lfloor \nicefrac{\left(n-i\right)p_{\kappa}}{\kappa}\right\rfloor +p_{\kappa}+p_{\kappa}p}^{s_{2}}}{M_{\left\lfloor \nicefrac{np_{\kappa}}{\kappa}\right\rfloor +p_{\kappa}+p_{\kappa}p}^{s_{2}}}\le\frac{M_{\left\lfloor \nicefrac{np_{\kappa}}{\kappa}\right\rfloor -p_{i}+p_{\kappa}+p_{\kappa}p}^{s_{2}}}{M_{\left\lfloor \nicefrac{np_{\kappa}}{\kappa}\right\rfloor +p_{\kappa}+p_{\kappa}p}^{s_{2}}}\le\frac{M_{p_{\kappa}p}^{s_{2}}}{M_{p_{\kappa}p+p_{i}}^{s_{2}}},
\]
and we can further estimate the right-hand side of the inequality
above to receive
\begin{align}
\left\Vert \partial_{m_{1},t}^{n}w_{p+1}(t,z)\right\Vert _{\tilde{r}} & \leq B'^{p}C'K'^{n+\kappa p+\kappa}BM_{\left\lfloor \nicefrac{np_{\kappa}}{\kappa}\right\rfloor +p_{\kappa}+p_{\kappa}p}^{s_{2}}\sum_{i\in\Kk\cup\left\{ 0\right\} }\sum_{q\in Q_{i}}\left(\frac{M_{p_{\kappa}p}}{M_{p_{\kappa}p+p_{i}}}\right)^{s_{2}}\partial_{m_{2},z}^{q-p_{\kappa}}(P_{p})(|z|).
\end{align}

Now let us notice that the set $\Kk\setminus\Kk'$ has no more than $\kappa$ elements. From this observation it follows that
\begin{multline*}
\sum_{i\in\Kk\cup\left\{ 0\right\} }\sum_{q\in Q_{i}}\left(\frac{M_{p_{\kappa}p}}{M_{p_{\kappa}p+p_{i}}}\right)^{s_{2}}\partial_{m_{2},z}^{q-p_{\kappa}}(P_{p})(|z|) \\
 \leq\left[(1+\kappa)\partial_{m_{2},z}^{-p_{\kappa}}+\sum_{i\in\Kk'}\sum_{q\in Q_{i}}\left(\frac{M_{p_{\kappa}p}}{M_{p_{\kappa}p+p_{i}}}\right)^{s_{2}}\partial_{m_{2},z}^{q-p_{\kappa}}\right](P_{p})(|z|)\\
 \leq(1+\kappa)\left[\partial_{m_{2},z}^{-p_{\kappa}}+\sum_{i\in\Kk'}\sum_{q\in Q_{i}}\left(\frac{M_{p_{\kappa}p}}{M_{p_{\kappa}p+p_{i}}}\right)^{s_{2}}\partial_{m_{2},z}^{q-p_{\kappa}}\right](P_{p})(|z|).
\end{multline*}
If we take $B'=(1+\kappa)B$ then
\begin{multline*}
\left\Vert \partial_{m_{1},t}^{n}w_{p+1}(t,z)\right\Vert _{\tilde{r}}\leq B'^{p+1}C'K'^{n+\kappa p+\kappa}M_{\left\lfloor \nicefrac{np_{\kappa}}{\kappa}\right\rfloor +p_{\kappa}+p_{\kappa}p}^{s_{2}}\\
\times \left[\partial_{m_{2},z}^{-p_{\kappa}}+\sum_{i\in\Kk'}\sum_{q\in Q_{i}}\left[\frac{M_{p_{\kappa}p}}{M_{p_{\kappa}p+p_{i}}}\right]^{s_{2}}\partial_{m_{2},z}^{q-p_{\kappa}}\right](P_{p})(|z|),
\end{multline*}
which concludes the proof of this lemma.
\end{proof}

\begin{proof}[Proof of Lemma~\ref{lema589}]
Inequality (\ref{eq:P_p-bound}) holds for $p=0$. Let us now fix
any $p\in\NN$. Then, we have

\begin{multline*}
P_{p}(z)=\left(\partial_{m_{2},z}^{-p_{\kappa}}+\sum_{i\in\Kk'}\sum_{q\in Q_{i}}\left[\frac{M_{p_{\kappa}\left(p-1\right)}}{M_{p_{\kappa}\left(p-1\right)+p_{i}}}\right]^{s_{2}}\partial_{m_{2},z}^{q-p_{\kappa}}\right)\dots\left(\partial_{m_{2},z}^{-p_{\kappa}}+\sum_{i\in\Kk'}\sum_{q\in Q_{i}}\left[\frac{M_{p_{\kappa}}}{M_{p_{\kappa}+p_{i}}}\right]^{s_{2}}\partial_{m_{2},z}^{q-p_{\kappa}}\right)\\
\left(\partial_{m_{2},z}^{-p_{\kappa}}+\sum_{i\in\Kk'}\sum_{q\in Q_{i}}\left[\frac{M_{0}}{M_{p_{i}}}\right]^{s_{2}}\partial_{m_{2},z}^{q-p_{\kappa}}\right)\left(P_{0}\right)(z),
\end{multline*}
which can be also rewritten in the form 
$$
P_{p}(z)=\left(\mathcal{P}(\partial_{m_2,z}^{-1})\right)P_0(z),
$$
with 
\begin{multline*}
\mathcal{P}(\partial_{m_2,z}^{-1}):=\partial_{m_{2},z}^{-p_{\kappa}p}\\
+\sum_{j=1}^{p}\sum_{1\leq l_{1}<l_{2}<\dots<l_{j}\leq p}\sum_{\left(i_{1},\dots,i_{j}\right)\in\Kk'^{j}}\sum_{q_{i_{1}}\in Q_{i_{1}},\dots,q_{i_{j}}\in Q_{i_{j}}}\prod_{s=1}^{j}\left[\frac{M_{p_{\kappa}\left(l_{s}-1\right)}}{M_{p_{\kappa}\left(l_{s}-1\right)+p_{i_{s}}}}\right]^{s_{2}}\partial_{m_{2},z}^{q_{i_{1}}+\dots+q_{i_{j}}-p_{\kappa}p}.
\end{multline*}

Hence, 
\begin{multline*}
P_{p}(|z|)=\frac{|z|^{p_{\kappa}p}}{m_{2}(p_{\kappa}p)}+\sum_{j=1}^{p}\sum_{1\leq l_{1}<l_{2}<\dots<l_{j}\leq p}\sum_{\left(i_{1},\dots,i_{j}\right)\in\Kk'^{j}}\sum_{q_{i_{1}}\in Q_{i_{1}},\dots,q_{i_{j}}\in Q_{i_{j}}}\prod_{s=1}^{j}\left[\frac{M_{p_{\kappa}\left(l_{s}-1\right)}}{M_{p_{\kappa}\left(l_{s}-1\right)+p_{i_{s}}}}\right]^{s_{2}}\\\times\frac{|z|^{p_{\kappa}p-\left(q_{i_{1}}+\dots+q_{i_{j}}\right)}}{m_{2}\left(p_{\kappa}p-\left(q_{i_{1}}+\dots+q_{i_{j}}\right)\right)}.
\end{multline*}
Directly from Remark \ref{rem-p_k-p_i} we conclude that $p_{i}<p_{\kappa}$
for $i<\kappa$. From this it follows that 
\[
p_{\kappa}p-\left(q_{i_{1}}+\dots+q_{i_{j}}\right)\geq p_{\kappa}p-\left(p_{i_{1}}+\dots+p_{i_{j}}\right)\ge p_{\kappa}p-\left(p_{\kappa}-1\right)p=p.
\]
We can use the above and the fact that $|z|<r<1$ to receive 
\begin{multline*}
P_{p}(|z|)\le\frac{|z|^{p}}{m_{2}(p_{\kappa}p)}\left(1+\sum_{j=1}^{p}\sum_{1\leq l_{1}<l_{2}<\dots<l_{j}\leq p}\sum_{\left(i_{1},\dots,i_{j}\right)\in\Kk'^{j}}\sum_{q_{i_{1}}\in Q_{i_{1}},\dots,q_{i_{j}}\in Q_{i_{j}}}\prod_{s=1}^{j}\left[\frac{M_{p_{\kappa}\left(l_{s}-1\right)}}{M_{p_{\kappa}\left(l_{s}-1\right)+p_{i_{s}}}}\right]^{s_{2}}\right.\\
\times \left.\frac{m_{2}\left(p_{\kappa}p\right)}{m_{2}\left(p_{\kappa}p-\left(q_{i_{1}}+\dots+q_{i_{j}}\right)\right)}\right)
\end{multline*}

Let us note that 
\begin{align*}
\frac{m_{2}\left(p_{\kappa}p\right)}{m_{2}\left(p_{\kappa}p-\left(q_{i_{1}}+\dots+q_{i_{j}}\right)\right)} & \leq A_{2}^{q_{i_{1}}+\dots+q_{i_{j}}}\left[\frac{M_{p_{\kappa}p}}{M_{p_{\kappa}p-q_{i_{1}}-\dots-q_{i_{j}}}}\right]^{s_{2}}\\
& \leq A_{2}^{p_{\kappa}j}\left[\frac{M_{p_{\kappa}p}}{M_{p_{\kappa}\left(p-j\right)+\left(p_{\kappa}-p_{i_{1}}\right)+\dots+\left(p_{\kappa}-p_{i_{j}}\right)}}\right]^{s_{2}}.
\end{align*}
What remains now, is to find an upper bound for
\[
\left[\frac{M_{p_{\kappa}p}}{M_{p_{\kappa}\left(p-j\right)+\left(p_{\kappa}-p_{i_{1}}\right)+\dots+\left(p_{\kappa}-p_{i_{j}}\right)}}\right]^{s_{2}}\prod_{s=1}^{j}\left[\frac{M_{p_{\kappa}\left(l_{s}-1\right)}}{M_{p_{\kappa}\left(l_{s}-1\right)+p_{i_{s}}}}\right]^{s_{2}}.
\]

First let us notice that 
\begin{equation*}
\prod_{s=1}^{j}\frac{M_{p_{\kappa}\left(l_{s}-1\right)}}{M_{p_{\kappa}\left(l_{s}-1\right)+p_{i_{s}}}}  \leq\prod_{s=1}^{j}\frac{M_{p_{\kappa}\left(s-1\right)}}{M_{p_{\kappa}\left(s-1\right)+p_{i_{s}}}}=\prod_{s=1}^{j}\frac{M_{p_{\kappa}\left(s-1\right)}}{M_{p_{\kappa}\left(s-1\right)+1}}\cdot\ldots\cdot\frac{M_{p_{\kappa}\left(s-1\right)+p_{i_{s}}-1}}{M_{p_{\kappa}\left(s-1\right)+p_{i_{s}}}}.
\end{equation*}
We would like to rewrite the right-hand side of this inequality as a product containing a term $\frac{1}{M_{p_\kappa j}}$. To this end notice that for $s=1,2,\dots,j$ we have
\begin{multline*}
\frac{M_{p_{\kappa}\left(s-1\right)}}{M_{p_{\kappa}\left(s-1\right)+1}}\cdot\ldots\cdot\frac{M_{p_{\kappa}\left(s-1\right)+p_{i_{s}}-1}}{M_{p_{\kappa}\left(s-1\right)+p_{i_{s}}}}=\left(\frac{M_{p_{\kappa}\left(s-1\right)}}{M_{p_{\kappa}\left(s-1\right)+1}}\cdot\ldots\cdot\frac{M_{p_{\kappa}\left(s-1\right)+p_{\kappa-1}}}{M_{p_{\kappa}\left(s-1\right)+p_{\kappa}}}\right)\\
\times\left(\frac{M_{p_{\kappa}\left(s-1\right)+p_{i_s}+1}}{M_{p_{\kappa}\left(s-1\right)+p_{i_s}}}\cdot\ldots\cdot\frac{M_{p_{\kappa}\left(s-1\right)+p_{i_s}+p_\kappa-p_{i_s}}}{M_{p_{\kappa}\left(s-1\right)+p_{i_s}+p_\kappa-p_{i_s}-1}}\right).
\end{multline*}
Using this fact we can write 
$$
\prod_{s=1}^{j}\frac{M_{p_{\kappa}\left(s-1\right)}}{M_{p_{\kappa}\left(s-1\right)+1}}\cdot\ldots\cdot\frac{M_{p_{\kappa}\left(s-1\right)+p_{i_{s}}-1}}{M_{p_{\kappa}\left(s-1\right)+p_{i_{s}}}}=\frac{M_{0}}{M_{p_{\kappa}j}}\prod_{s=1}^{j}\frac{M_{p_{\kappa}\left(s-1\right)+p_{i_{s}}+1}}{M_{p_{\kappa}\left(s-1\right)+p_{i_{s}}}}\cdot\dots\cdot\frac{M_{p_{\kappa}\left(s-1\right)+p_{i_{s}}+p_{\kappa}-p_{i_{s}}}}{M_{p_{\kappa}\left(s-1\right)+p_{i_{s}}+p_{\kappa}-p_{i_{s}}}}
$$

Moreover, we have 
\begin{multline*}
\frac{1}{M_{p_{\kappa}\left(p-j\right)+\left(p_{\kappa}-p_{i_{1}}\right)+\dots+\left(p_{\kappa}-p_{i_{j}}\right)}}  =\frac{1}{M_{p_{\kappa}\left(p-j\right)}}\cdot\frac{M_{p_{\kappa}\left(p-j\right)}}{M_{p_{\kappa}\left(p-j\right)+\left(p_{\kappa}-p_{i_{1}}\right)+\dots+\left(p_{\kappa}-p_{i_{j}}\right)}}\\
 \leq\frac{1}{M_{p_{\kappa}\left(p-j\right)}}\cdot\frac{M_{0}}{M_{\left(p_{\kappa}-p_{i_{1}}\right)+\dots+\left(p_{\kappa}-p_{i_{j}}\right)}}\\
 =\frac{1}{M_{p_{\kappa}\left(p-j\right)}}\prod_{s=1}^{j}\frac{M_{\left(p_{\kappa}-p_{i_{1}}\right)+\dots+\left(p_{\kappa}-p_{i_{s-1}}\right)}}{M_{\left(p_{\kappa}-p_{i_{1}}\right)+\dots+\left(p_{\kappa}-p_{i_{s-1}}\right)+1}}\cdot\dots\cdot\frac{M_{\left(p_{\kappa}-p_{i_{1}}\right)+\dots+\left(p_{\kappa}-p_{i_{s-1}}\right)+p_{\kappa}-p_{i_{s}}-1}}{M_{\left(p_{\kappa}-p_{i_{1}}\right)+\dots+\left(p_{\kappa}-p_{i_{s-1}}\right)+p_{\kappa}-p_{i_{s}}}}.
\end{multline*}
We can therefore write 
\[
\left[\frac{M_{p_{\kappa}p}}{M_{p_{\kappa}\left(p-j\right)+\left(p_{\kappa}-p_{i_{1}}\right)+\dots+\left(p_{\kappa}-p_{i_{j}}\right)}}\right]^{s_{2}}\prod_{s=1}^{j}\left[\frac{M_{p_{\kappa}\left(l_{s}-1\right)}}{M_{p_{\kappa}\left(l_{s}-1\right)+p_{i_{s}}}}\right]^{s_{2}}\le\left(\frac{M_{p_{\kappa}p}}{M_{p_{\kappa}j}M_{p_{\kappa}(p-j)}}\right)^{s_{2}}\prod_{s=1}^{j}B_{s}^{s_{2}}
\]
with 
\[
B_{s}=\left\{ \begin{array}{rl}
1 & \quad\textrm{ for }i_{s}=\kappa\\
{\displaystyle \prod_{\sigma=1}^{p_{\kappa}-p_{i_{s}}}\frac{M_{\left(p_{\kappa}-p_{i_{1}}\right)+\dots+\left(p_{\kappa}-p_{i_{s-1}}\right)+\sigma-1}M_{p_{\kappa}\left(s-1\right)+p_{i_{s}}+\sigma}}{M_{\left(p_{\kappa}-p_{i_{1}}\right)+\dots+\left(p_{\kappa}-p_{i_{s-1}}\right)+\sigma}M_{p_{\kappa}\left(s-1\right)+p_{i_{s}}+\sigma-1}}} & \quad\textrm{ for }i_{s}<\kappa.
\end{array}\right.
\]
If $i_{s}<\kappa$ then from (\ref{eq:seq_property}) and Lemma \ref{lem:decreasing}
it follows that
\[
B_{s}\leq\prod_{\sigma=1}^{p_{\kappa}-p_{i_{s}}}\frac{M_{s-1+\sigma-1}}{M_{s-1+\sigma}}\frac{M_{p_{\kappa}s-p_{\kappa}+p_{i_{s}}+\sigma}}{M_{p_{\kappa}s-p_{\kappa}+p_{i_{s}}+\sigma-1}}\leq\prod_{\sigma=1}^{p_{\kappa}-p_{i_{s}}}\frac{M_{s-1}}{M_{s}}\frac{M_{p_{\kappa}s}}{M_{p_{\kappa}s-1}}\leq C_{3}(p_{\kappa})^{p_{\kappa}-p_{i_{s}}}.
\]

Thus, one can upper estimate $P_{p}(|z|)$ by

\begin{align*}
 & \phantom{\le} \left(1+\sum_{j=1}^{p}\binom{p}{j}
\sum_{\substack{\left(i_{1},\dots,i_{j}\right)\in\Kk'^{j} \\ q_{i_{1}}\in Q_{i_{1}},\dots,q_{i_{j}}\in Q_{i_{j}}}}
A_{2}^{p_{\kappa}j}\left(\frac{M_{p_{\kappa}p}}{M_{p_{\kappa}j}M_{p_{\kappa}(p-j)}}\right)^{s_{2}}\prod_{s=1}^{j}C_{3}(p_{\kappa})^{s_{2}(p_{\kappa}-p_{i_{s}})}\right)\frac{|z|^{p}}{m_{2}(p_{\kappa}p)}\\
 & \le\left(1+\sum_{j=1}^{p}\binom{p}{j}
\sum_{\substack{\left(i_{1},\dots,i_{j}\right)\in\Kk'^{j} \\ q_{i_{1}}\in Q_{i_{1}},\dots,q_{i_{j}}\in Q_{i_{j}}}}
A_{2}^{p_{\kappa}j}\left(\frac{M_{p_{\kappa}p}}{M_{p_{\kappa}j}M_{p_{\kappa}(p-j)}}\right)^{s_{2}}C_{3}(p_{\kappa})^{s_{2}(p_{\kappa}-1)j}\right)\frac{|z|^{p}}{m_{2}(p_{\kappa}p)}\\
 & =\left(\sum_{j=0}^{p}\binom{p}{j}A_{2}^{p_{\kappa}j}\left(\frac{M_{p_{\kappa}p}}{M_{p_{\kappa}j}M_{p_{\kappa}(p-j)}}\right)^{s_{2}}C_{3}(p_{\kappa})^{s_{2}(p_{\kappa}-1)j}
\sum_{\substack{\left(i_{1},\dots,i_{j}\right)\in\Kk'^{j} \\ q_{i_{1}}\in Q_{i_{1}},\dots,q_{i_{j}}\in Q_{i_{j}}}}
1\right)\frac{|z|^{p}}{m_{2}(p_{\kappa}p)}.
\end{align*}
In order to find an upper bound of 
\[
\sum_{\substack{\left(i_{1},\dots,i_{j}\right)\in\Kk'^{j} \\ q_{i_{1}}\in Q_{i_{1}},\dots,q_{i_{j}}\in Q_{i_{j}}}}
1
\]
it remains to notice that $|\Kk'^{j}|\le\kappa^{j}$ and $|Q_{i}|\le p_{\kappa}$
for every $i\in\Kk'.$ From these observations and condition (mg)
it follows that 
\[
P_{p}(|z|)\le\tilde{A}_{1}^{p_{\kappa}p}\left(1+A_{2}^{p_{\kappa}}C_{3}(p_{\kappa})^{s_{2}(p_{\kappa}-1)}\kappa p_{\kappa}\right)^{p}\frac{|z|^{p}}{m_{2}(p_{\kappa}p)}
\]
for certain constant $\tilde{A}_{1}>0$. The result follows from here.
\end{proof}



\begin{thebibliography}{99}
\bibitem{balser} W. Balser, Formal power series and linear systems of meromorphic ordinary differential equations, Universitext, Springer-Verlag, New York, 2000.
\bibitem{BY} W. Balser, M. Yoshino, \textit{Gevrey order of formal power series solutions of inhomogeneous partial differential equations with constant coefficients}, Funkcial. Ekvac. 53 (2010), 411--434.
\bibitem{immink} G.K. Immink, \emph{Exact asymptotics of nonlinear difference equations with levels 1 and 1+,} Ann. Fac. Sci. Toulouse T.XVII (2) (2008), 309--356.
\bibitem{immink2} G. K. Immink, \emph{Accelero-summation of the formal solutions of nonlinear difference equations,} Ann. Inst. Fourier (Grenoble) 61 (2011), no. 1, 1--51.
\bibitem{jjs17} J. Jim\'enez-Garrido, J. Sanz, G. Schindl, \emph{Log-convex sequences and nonzero proximate orders,} J. Math. Anal. Appl. 448(2) (2017) 1572--1599.
\bibitem{kilbas10} A. Kilbas, \emph{Partial fractional differential equations and some of their applications,} Analysis, Munchen 30 (2010), no. 1, 35--66.
\bibitem{kilbasetal06} A. Kilbas, H. Srivastava, J. Trujillo. Theory and applications of fractional differential equations. North-Holland Math. Stud. 204, Elsevier, Amsterdam, 2006.
\bibitem{lastramaleksanz} A. Lastra, S. Malek, J. Sanz, \emph{Summability in general Carleman ultraholomorphic classes}, J. Math. Anal. Appl. 430 (2015), 1175--1206. 
\bibitem{lastramaleksanz2} A. Lastra, S. Malek, J. Sanz, \emph{Strongly regular multi-level solutions of singularly perturbed linear partial differential equations}, Results. Math. 70 (2016), 581--614.
\bibitem{LMS} A. Lastra, S. Michalik, M. Suwi\'nska, \emph{Estimates of formal solutions for some generalized moment partial differential equations}, 2020, submitted (\href{https://arxiv.org/abs/1911.11998}{arXiv:1911.11998}).
\bibitem{LMS2}  A. Lastra, S. Michalik, M. Suwi\'nska, \emph{Summability of formal solutions for some generalized moment partial differential equations}, to appear in Results. Math., 2021.
\bibitem{michalik10} S. Michalik, \textit{Summability and fractional linear partial differential equations}, J. Dyn. Control Syst. 16(4) (2010), 557--584. 
\bibitem{michalik12} S. Michalik, \textit{Multisummability of formal solutions of inhomogeneous linear partial differential equations with constant coefficients}, J. Dyn. Control Syst. 18 (2012), 103--133.
\bibitem{M} S. Michalik, \textit{Analytic solutions of moment partial differential equations with constant coefficients}, Funkcial. Ekvac. 56 (2013), 19--50.
\bibitem{michalik13jmaa} S. Michalik, \textit{Summability of formal solutions of linear partial differential equations with divergent initial data}, J. Math. Anal. Appl. 406 (2013), 243--260.
\bibitem{michalik17fe} S. Michalik, \textit{Analytic summable solutions of inhomogeneous moment partial differential equations}, Funkcial. Ekvac. 60 (2017), 325--351.
\bibitem{michaliksuwinska}  S. Michalik, M. Suwi\'nska, \textit{Gevrey estimates for certain moment partial differential equations}, Complex Differential and Difference Equations, De Gruyter Proceedings in Mathematics (2020), 391--408.
\bibitem{michaliktkacz} S. Michalik, B. Tkacz, \emph{The Stokes Phenomenon for some moment partial differential equations}, J. Dyn. Control Syst. 25 (2019),  573--598. 
\bibitem{remy2016} P. Remy, \emph{Gevrey order and summability of formal series solutions of some classes of inhomogeneous linear partial differential equations with variable coefficients}. J. Dyn. Control Syst. 22 (2016), 693--711. 
\bibitem{remy2017} P. Remy, \emph{Gevrey order and summability of formal series solutions of certain classes of inhomogeneous linear integro-differential equations with variable coefficients}. J. Dyn. Control Syst. 23 (2017), 853--878.
\bibitem{sanzproceedings} J. Sanz, \emph{Asymptotic analysis and summability of formal power series}, Analytic, algebraic and geometric aspects of differential equations, Trends Math., Birkh\"auser/Springer, Cham, (2017) 199--262.
\bibitem{suwinska} M. Suwi\'nska, \textit{Gevrey estimates of formal solutions for certain moment partial differential equations with variable coefficients},  J. Dyn. Control Syst. (2020). \href{https://doi.org/10.1007/s10883-020-09504-3}{https://doi.org/10.1007/s10883-020-09504-3}
\bibitem{thilliez} V. Thilliez, \emph{Division by flat ultradifferentiable functions and sectorial extensions}, Results Math. 44 (2003), 169--188.
\bibitem{yonemura} A. Yonemura, \emph{Newton polygons and formal Gevrey classes}, Publ. RIMS Kyoto Univ. 1990; 26:197--204.
\end{thebibliography}
\end{document}